\documentclass[]{sn-jnl}

\usepackage{graphicx}%
\usepackage{multirow}%
\usepackage{amsmath,amssymb,amsfonts}%
\usepackage{amsthm}%
\usepackage{mathrsfs}%
\usepackage[title]{appendix}%
\usepackage{xcolor}%
\usepackage{textcomp}%
\usepackage{manyfoot}%
\usepackage{booktabs}%
\usepackage{algorithm}%
\usepackage{algorithmicx}%
\usepackage{algpseudocode}%
\usepackage{listings,enumerate,cite}%
\usepackage{tikz}
\usepackage[left,pagewise,mathlines]{lineno}
\usepackage{amsmath}

\newtheoremstyle{compact}
{3pt}                         
{3pt}                         
{\normalfont\normalsize}       
{}
{\bfseries}
{.}
{0.5em}
{}

\theoremstyle{compact}

\newtheorem{theorem}{Theorem}
\newtheorem{proposition}{Proposition}
\newtheorem{lemma}{Lemma}
\newtheorem{corollary}{Corollary}
\newtheorem{remark}[theorem]{Remark}
\newtheorem{example}[theorem]{Example}
\usepackage{etoolbox}
\AtBeginEnvironment{proof}{%
	\normalsize
	\setlength{\topsep}{3pt}%
}

\usepackage{titlesec}
\titleformat{\section}
{\normalfont\normalsize\bfseries}
{\thesection}{1em}{}

\theoremstyle{thmstylethree}%
\newcommand{\vertiii}[1]{\left\vert\kern-0.25ex\left\vert\kern-0.25ex\left\vert #1\right\vert\kern-0.25ex\right\vert\kern-0.25ex\right\vert}
\newcommand {\dom} {{\rm dom}\,}

\newcommand{\al}{\alpha}
\newcommand{\be}{\beta}

\newcommand{\bx}{\bar x}

\newcommand {\R} {\mathbb R}
\newcommand {\N} {\mathbb N}

\newcommand {\sd} {\partial}

\newcommand {\B} {\mathbb B}

\usepackage[left,pagewise]{lineno}
\usepackage{etoolbox}

\newcommand*\patchMathEnvironmentForLineno[1]{%
	\BeforeBeginEnvironment{#1}{\begin{linenomath}}%
		\AfterEndEnvironment{#1}{\end{linenomath}}%
}

\patchMathEnvironmentForLineno{align}
\patchMathEnvironmentForLineno{align*}
\patchMathEnvironmentForLineno{gather}
\patchMathEnvironmentForLineno{gather*}
\patchMathEnvironmentForLineno{equation}
\patchMathEnvironmentForLineno{equation*}
\patchMathEnvironmentForLineno{multline}
\patchMathEnvironmentForLineno{multline*}

\newcommand*\fixAMSLineNumbers[1]{%
	\AtBeginEnvironment{#1}{\postdisplaypenalty=0\relax}%
}

\fixAMSLineNumbers{gather}
\fixAMSLineNumbers{gather*}
\fixAMSLineNumbers{align}
\fixAMSLineNumbers{align*}

\usepackage{mathptmx}   
\usepackage[colorinlistoftodos]{todonotes}
\hypersetup{
	colorlinks=true,
	linkcolor=blue, 
	citecolor=red, 
	urlcolor=blue  } 
\usetikzlibrary{patterns,arrows.meta}

\begin{document}

\title[Optimality Conditions and Solution Set Characterisations of Convex Minimal Time Optimisation Problems]{Optimality Conditions and Solution Set Characterisations of Convex Minimal Time Optimisation Problems}

\author[1,2]{\fnm{Doan} \sur{Huu Hieu}}\email{dhhieu.math@gmail.com}

\author*[3]{\fnm{Nguyen} \sur{Duy Cuong}}\email{ndcuong@ctu.edu.vn}

\affil[1]{\orgdiv{Faculty of Mathematics and Computer Science}, \orgname{University of Science}, \city{Ho Chi Minh City},  \country{Vietnam}}

\affil[2]{\orgdiv{Vietnam National University}, \city{Ho Chi Minh City},  \country{Vietnam}}

\affil[3]{\orgdiv{Faculty of Mathematics}, \orgname{College of Natural Sciences, Can Tho University}, \city{Can Tho City},  \country{Vietnam}}

\abstract{
The paper studies convex minimisation problems involving a composition of a coordinatewise nondecreasing convex function with  a finite family of nonnegative, locally Lipschitz continuous and convex functions.
We establish complete dual optimality conditions and characterise the entire solution set using multipliers and subgradients associated with a given solution. 
The results are applied to norm optimisation problems involving minimal time functions with different convex dynamics sets and convex target sets.
The sum, maximum and $p$-norms yield minimal time versions of the  Fermat--Torricelli, Chebyshev centre and $p$-Fermat--Torricelli problems, respectively.
Examples in finite and infinite dimensional spaces are provided to illustrate the solution set characterisations.
}


\maketitle



\section{Introduction}

Given a finite family of nonempty closed convex sets in a normed space, the generalised Fermat–Torricelli problem seeks a point that minimises the sum of distances to these sets, while the generalised Chebyshev centre problem minimises the maximum distance.
When the sets are singletons, they reduce to the classical Fermat–Torricelli and Chebyshev centre problems.
These problems play an important role in facility location theory and related areas and have been extensively studied from both theoretical and computational viewpoints \cite{BolMarSol99,HieCuo26,NamAnRecSun14,NamNguSal12,BotTur94,NamHoa13,Cuo26b,MorNam11+}.

Distance functions measure proximity to target sets by taking the infimum of the distances from a given point to points in each set and use the same norm for all targets.
In practice, however, the cost of reaching a target may depend on the direction of movement or on the target itself.
Minimal time functions provide a natural extension that accommodates these features.
Given a target set $\Omega$ and a dynamics set $F$, the minimal time function assigns to a point $x$ the infimum of all $t\ge0$ for which $x+tF$ intersects $\Omega$ \cite{MorNam11,ColWol04a}.
It measures the least time needed to reach the target under the movement constraints described by $F$ and reduces to the usual distance function when $F$ is the closed unit ball.

Minimal time functions can serve as cost functions in location problems by measuring the time needed to reach each target. Allowing different dynamics sets for different targets permits these costs to vary among targets, while non-symmetric dynamics sets  allow costs to depend on the direction of movement. These features extend location models based on distances.

Optimisation problems involving minimal time functions have been studied in several settings. In \cite{MorNam11+}, the authors studied the Fermat--Torricelli problem with a common bounded dynamics set and derived optimality conditions in convex and nonconvex settings. In \cite{NamVilAn12}, the authors considered the Chebyshev centre problem with a common, possibly unbounded dynamics set and proposed a projected subgradient algorithm in finite-dimensional spaces. Directional minimal time functions and their applications to directional and vectorial Fermat--Torricelli problems were also studied in \cite{DurPanStr16,NamZal13,DurPanStr17}.
Minimal time functions play an important role in optimal control, where they describe the least time needed to reach a target along admissible trajectories \cite{Bar89,WolZhu98}. Other research directions include necessary and sufficient conditions for error bounds of minimal time functions \cite{Long21,Long22} and well-posedness and regularity properties of minimal time problems in finite and infinite dimensional spaces \cite{ColGonMor10}.

Generalised differentiation of minimal time functions in convex and nonconvex settings provides a foundation for these studies. 
A main focus is to establish formulas and calculus rules for Fréchet and limiting subdifferentials.
In particular, subdifferentials of minimal time functions are described in terms of normal cones to the target sets, polar sets of the dynamics sets and subdifferentials of the associated Minkowski functionals. 
The reader is referred to \cite{MorNam11,NamVilAn12,NamZal13,HeNg06,ColWol04a,ColWol04b,ColGonMor10,ZhaHeJia14,JiaHe09} and the references therein for further discussions and results.

To the best of our knowledge, a general framework for dual optimality conditions and solution set representations covering both minimal time Fermat--Torricelli and Chebyshev centre problems has not been established. 
The present paper continues our work on location problems involving distances to points and convex sets \cite{Cuo26b,HieCuo26} by developing such a unified framework for convex minimisation problems involving minimal time functions. 
The convexity allows us to establish necessary and sufficient optimality conditions and characterise the entire solution set using dual data associated with a given solution. In the nonconvex setting, subdifferential conditions generally provide only necessary conditions for local optimality and do not guarantee global optimality. The nonconvex setting lies beyond the scope of this paper.

We first study a general convex composition minimisation problem in which 
a finite family of nonnegative, locally Lipschitz continuous and convex functions
is combined through a convex function which is coordinatewise nondecreasing on the nonnegative orthant of $\R^n$.
Employing a subdifferential chain rule studied in \cite{Cla90} and recently refined in \cite{HieCuo26}, we establish complete dual necessary and sufficient optimality conditions.
Given a solution and dual data arising from the optimality conditions, the entire solution set is characterised via a fixed outer multiplier and  component subgradients.

Characterising the entire solution set is an important topic in optimisation, particularly in location problems, where restrictions may prevent the use of a particular optimal point and a description of all solutions allows one to choose another candidate without increasing the objective value.
Such a description also provides a foundation for studying geometric properties of the solution set and its behaviour under perturbations.
Solution set characterisations for differentiable convex programmes were studied by Mangasarian \cite{Man88} and subsequently extended by Burke and Ferris \cite{Bur91} to nonsmooth convex programmes.
In particular, \cite[Lemma 2(a) \& Theorem 5]{Bur91} shows that the set of subgradients satisfying the optimality condition is independent of the chosen solution.
Our result in this paper complements these characterisations by describing the solution set in terms of the individual components and the outer function of the composition.

Building on this general framework, we study
norm optimisation problems involving minimal time functions associated with nonempty closed convex target sets and bounded closed convex dynamics sets containing the origin in their interiors.
Different norms reflect different priorities in balancing the costs of reaching the targets.
The norms used to combine the minimal time functions are generated by a class of continuous convex functions defined on the standard simplex in $\mathbb R^n$ \cite{SaiKatTak00,Cuo26a}.
This construction provides explicit representations of the dual norms in terms of the primal generating functions and includes the sum, maximum and $p$-norms. These choices yield minimal time versions of the Fermat--Torricelli, Chebyshev centre and $p$-Fermat--Torricelli problems, respectively.
Comparisons with existing results are discussed whenever applicable.
Examples in finite and infinite dimensional spaces equipped with different norms are provided to illustrate the established results.

The rest of the paper is organised as follows. Section \ref{S2} collects notation and preliminary results used throughout the paper. 
Section \ref{S3} establishes dual optimality conditions and solution set descriptions for a general convex composition minimisation problem.
Sufficient conditions for existence and boundedness of the solution set are also discussed.
Section \ref{S4} applies these results to norm optimisation problems involving minimal time functions.
Sections \ref{S5}--\ref{S7} study the sum, maximum and $p$-norm formulations, respectively.
The final Section~\ref{S8} summarises the contribution of the paper and outlines potential directions for future research.
 
\section{Preliminaries}\label{S2}
Our basic notation is standard, see, e.g., \cite{Mor06.1,RocWet98,Roc70}.
Throughout the paper, if not explicitly stated otherwise, $X$ is a vector space equipped with a norm $\|\cdot\|$.
The topological dual of a normed space $X$ is denoted by $X^*$, while $\langle\cdot,\cdot\rangle$ and $\|\cdot\|^*$ denote the bilinear pairing between the two spaces and the dual norm, respectively.
The zero vector and closed unit ball in $X$ are denoted by $0_X$ and $\B_X$, respectively.
When $X$ is a Hilbert space, its norm is given by $\|x\|:=\sqrt{\langle x,x\rangle}$ for all $x\in X$, and we identify $X$ with $X^*$ via the Riesz representation theorem \cite[Theorem~5.5]{Bre11}.
The linear subspace generated by a set $\Omega\subset X$ is denoted by $\operatorname{span}\Omega$. 
Symbols $\mathbb{N}$, $\mathbb{R}$, $\mathbb{R}_+$, and $\mathbb{R}_+^n$ denote the set of all positive integers, real numbers,  nonnegative real numbers, and the set of all vectors in $\mathbb{R}^n$ with nonnegative components, respectively.
We write $\infty$ instead of $+\infty$ and employ the convention that $\frac{1}{\infty}=0$ and $\frac{1}{0}=\infty$.
The distance from a point $\bx\in X$ to $\Omega$ is defined by $d(\bx,\Omega):=\inf_{x\in\Omega}\|\bx-x\|$.
If $\Omega=\emptyset$, we set $d(\bx,\Omega):=\infty$.
The positive part of a number $\al\in\R$ is defined by $\al^+:=\max\{\al,0\}$.
For a vector $u:=(u_1,\ldots,u_n)\in\R^n$, we denote $u^+:=(u_1^+,\ldots,u_n^+)$.

For a convex function $f:X\to\R\cup\{\infty\}$,
its subdifferential at $\bx\in\dom f:=\{x \in X\mid f(x) <\infty\}$ is given by  
\begin{gather*}
\sd f(\bx)=\left\{x^* \in X^*\mid \langle x^*,x-\bx\rangle\le f(x)-f(\bx)\;\;\text{for all}\;x\in X\right\}.
\end{gather*}
It is convex and weak$^*$ closed.
We set $\partial{f}(\bx):=\emptyset$ if $\bx\notin\dom f$.
It is well known \cite[Example~3.36]{MorNam22} that
\begin{equation}\label{sn}
\partial\|\cdot\|(\bx) = 
\begin{cases}
\left\{x^*\in X^*\mid \|x^*\|^* =1\;\text{and}\;\langle x^*,\bx\rangle=\|\bx\|\right\} & \text{if}\;\bx \ne 0_X,\\
\B_{X^*}
 & \text{otherwise}.
\end{cases}
\end{equation}	

Given a nonempty convex subset $\Omega$ of $X$, the normal cone to $\Omega$ at $\bx\in\Omega$ is defined by
\begin{gather}\label{CN}
N_\Omega(\bx):=\{x^*\in X^*\mid\langle x^*,x-\bx\rangle\le 0\;\;\text{for all}\;x\in\Omega\}.
\end{gather}
If $\bx\notin\Omega$, we set $N_\Omega(\bx):=\emptyset$.

The following statement gives a chain rule for the subdifferential of a convex composition; see
\cite[Theorem~2.3.9(i) \& Proposition~2.36(b)]{Cla90} and \cite{HieCuo26}.

\begin{lemma}\label{L1.2}
Let $\bar x\in X$, $f_i:X\to\R$ $(i=1,\ldots,n)$ be locally Lipschitz continuous at $\bx$, $f:=(f_1,\ldots,f_n)$, and
$\Phi:\R^n\to\R$.
Suppose that $f_1,\ldots,f_n,\Phi$ and $\Phi\circ f$  are convex, and $\al_1,\ldots,\al_n\ge 0$ for any $(\al_1,\ldots,\al_n)\in\partial \Phi(f(\bx))$.
Then
\begin{gather*}
\partial(\Phi\circ f)(\bx)= \left\{\sum_{i=1}^{n}\al_ix^*_i\Bigm|(\al_1,\ldots,\al_n)\in\partial\Phi(f(\bx)),\ x_i^*\in\partial f_i(\bx)\; (i=1,\ldots,n)\right\}.
\end{gather*}	
\end{lemma}	



From now on, we assume that
$F\subset X$ is a nonempty,  closed, bounded and convex  set with $0_X\in\operatorname{int}F$.
The \emph{Minkowski functional}, \emph{support function}, and   \emph{polar set}  of $F$ are defined, respectively, by
$\rho_F(x):=\inf\{t\in\R_+\mid x\in tF\}$ for all $x\in X$, $\sigma_F(x^*):=\sup_{x\in F}\langle x^*,x\rangle$ for all $x^*\in X^*$, and $F^\circ:=\{x^*\in X^*\mid \sigma_F(x^*)\le 1\}$.
If $F=\B_X$, then $\rho_F=\|\cdot\|$,
$\sigma_F=\|\cdot\|^*$, and $F^\circ=\B_{X^*}$.
It is well known that $\sigma_F$ is convex and positively homogeneous; $F^\circ$ is weak$^*$ closed and convex;
$\rho_F$ and $\sigma_F$ are finite-valued; $\rho_F$ is sublinear; $\rho_{F^\circ}(x^*)=\sigma_{F}(x^*)$ for all $x^*\in X^*$; $\sigma_F(x^*)=0$ if and only if $x^*=0_{X^*}$; $\rho_F(x)=0$ if and only if $x=0_X$ \cite{Pen13,MorNam22,MorNam11}.
The convex subdifferential of the Minkowski functional of $F$  \cite[Exercise 3.108]{MorNam22} is given by
\begin{equation}\label{L2.2-1}
\partial\rho_F(\bx)= 
\begin{cases}
\left\{x^*\in X^*\mid \sigma_{F}(x^*)=1,\;\langle x^*,\bx\rangle=\rho_F(\bx)\right\} & \text{\rm if}\;\;\bx \ne 0_X,\\
F^\circ & \text{\rm otherwise}.
\end{cases}
\end{equation}

Let $\Omega\subset X$ be a nonempty, closed and convex set.
The minimal time function 
 is defined by
\begin{gather*}
\mathcal{T}^F_\Omega(x):=\inf\{t\in\R_+\mid \Omega\cap(x+tF)\ne\emptyset\}\;\;\text{for all }x\in X.
\end{gather*}	
If $F=\B_X$, then $\mathcal{T}^F_\Omega(x)=d(x,\Omega)$ for all $x\in X$.
By the standing assumption, $\mathcal{T}^F_\Omega$ is finite-valued, convex and globally Lipschitz continuous; $\mathcal{T}^F_\Omega(x)=0$ if and only if $x\in\Omega$; and by \cite[Proposition~3.3]{MorNam11},
\begin{gather}\label{mi}
\mathcal{T}^F_\Omega(x)=\inf_{\omega\in\Omega}\rho_F(\omega-x)\;\;\text{for any } x\in X.
\end{gather}	
The minimal time projection  of $x\in X$ to $\Omega$ is defined by
\begin{gather}\label{mp}
P^F_\Omega(x):=(x+\mathcal{T}^F_\Omega(x)F)\cap\Omega.
\end{gather}	
It is nonempty if $\Omega$ is compact \cite{MorNam11}.
The convex subdifferential of the minimal time function  \cite[Theorems 7.1(i) \& 7.3]{MorNam11} is given by
\begin{equation}\label{sm}
\partial \mathcal{T}^F_\Omega(\bx)= 
\begin{cases}
N_\Omega(\bx)\cap (-F^\circ)& \text{\rm if}\;\;\bx\in\Omega,\\
N_\Omega(x)\cap\left(-\partial\rho_F(x-\bx)\right)& \text{if}\;\; \bx\notin \Omega,
\end{cases}
\end{equation}
where, in the second case,  we assume that
$P^F_\Omega(\bx)\ne\emptyset$ and $x$ is any point in \(P^F_\Omega(\bar x)\).

\begin{lemma}\label{L2.7}
The following assertions hold.
\begin{enumerate}[\rm (i)]
\item\label{L2.7-1}
$\langle x^*,x\rangle\le\sigma_F(x^*)\cdot\rho_F(x)$ for all 
$x\in X$ and $x^*\in X^*$.
\item\label{L2.7-2}
Let $x\in \Omega$, $0_{X^*}\ne x^*\in N_\Omega(x)$, $ 0_X\ne z\in X$, and $\langle x^*,z\rangle=\sigma_F(-x^*)\cdot\rho_F(-z)$.
Then $x+z\notin\Omega$ and $x\in P^F_\Omega(x+z)$.
\end{enumerate}
\end{lemma}	

\begin{proof}[\normalsize Proof]\normalsize
\begin{enumerate}[\rm (i)]
\item 
We first prove that $x\in\rho_F(x)F$ for all $x\in X$.
If $x=0_X$, then the statement is obvious. 
Suppose that $x\ne 0_X$ and $\xi:=\rho_F(x)$.
Note that $\xi\in (0,\infty)$.
Then there exists a sequence $\{t_k\}\subset\R_+$ such that $x\in t_kF$, $t_k\ge\xi$ $(k\in\N)$ and $t_k\to\xi$ as
$k\to\infty$.
Set $u_k:=x/t_k$ $(k\in\N)$.
Then $\{u_k\}\subset F$ and $u_k\to x/\xi$ as $k\to\infty$.
Since $F$ is closed, we have $x\in\xi F$.
Let $x\in X$ and $x^*\in X^*$.
Then $x=\rho_F(x)u$ for some $u\in F$, and consequently,
\begin{gather*}
\langle x^*,x\rangle=\langle x^*,\rho_F(x) u\rangle=\rho_F(x)\langle x^*,u\rangle
\le\sigma_F(x^*)\cdot\rho_F(x).
\end{gather*}
\item 
If $x+z\in\Omega$, then 
\begin{gather*}
 0<\sigma_F(-x^*)\cdot\rho_F(-z)=\langle x^*,z\rangle= \langle x^*,(x+z)-x\rangle\overset{\eqref{CN}}{\le} 0,
 \end{gather*}
a contradiction.
Thus, $x+z\notin\Omega$.
For any $\omega\in\Omega$, by assertion \eqref{L2.7-1},
\begin{align*}
\sigma_F(-x^*)\cdot\rho_F(\omega-(x+z))
&\ge\langle -x^*,\omega-(x+z)\rangle\\
&=\langle x^*,z\rangle+\langle x^*,x-\omega\rangle\\
&\overset{\eqref{CN}}{\ge} \langle x^*,z\rangle=\sigma_F(-x^*)\cdot\rho_F(-z).
\end{align*}	
Then $\rho_F(\omega-(x+z))\ge\rho_F(-z)$ since
$\sigma_F(-x^*)>0$.
By \eqref{mi},
\begin{gather*}
\mathcal{T}^F_\Omega(x+z)=\inf_{\omega\in\Omega}\rho_F(\omega-(x+z)).
\end{gather*}	
Thus, 
\begin{gather*}
\rho_F(-z)\le \mathcal{T}^F_\Omega(x+z)\le \rho_F(x-(x+z))=\rho_F(-z),
\end{gather*}
and consequently,
$\mathcal{T}^F_\Omega(x+z)=\rho_F(-z)$.
Note that $-z\in \rho_F(-z)F$.
Then
\begin{gather*}
x=(x+z)-z\in (x+z+\rho_F(-z)F)\cap\Omega=(x+z+\mathcal{T}^F_\Omega(x+z)F)\cap\Omega.
\end{gather*}
Hence, $x\in P^F_\Omega(x+z)$.
\end{enumerate}
\end{proof}

Let $n\ge 2$ and 
$\Delta_n:=\{(t_1,\ldots,t_{n})\in\R^{n}_+\mid \sum_{i=1}^nt_i =1\}$.
Denote by $\mathbf{e}_1,\ldots,\mathbf{e}_n$ the standard basis vectors of $\R^n$ and $\pmb{\Psi}_n$ the family of all convex continuous functions $\psi:\Delta_n\to\R$ such that $\psi(\mathbf{e}_1)=\cdots=\psi(\mathbf{e}_n)=1$ and
\begin{gather*}
\psi(t_1,\ldots,t_n)\ge (1-t_i)\cdot\psi\left(\dfrac{t_1}{1-t_i},\ldots,\dfrac{t_{i-1}}{1-t_i},0,\dfrac{t_{i+1}}{1-t_i},\ldots,\dfrac{t_{n}}{1-t_i}\right)
\end{gather*}	
for all $(t_1,\ldots,t_n)\in\Delta_n$ with $t_i<1$ 
$(i=1,\ldots,n)$.
A function $\psi\in \pmb{\Psi}_n$ is called \textit{permutation-symmetric} if $\psi(t_1,\ldots,t_n)=\psi(t_{\sigma(1)},\ldots,t_{\sigma(n)})$ for all $(t_1,\ldots,t_n)\in\Delta_n$ and every permutation
$\sigma$ of $\{1,\ldots,n\}$.

The following facts can be found in \cite{Cuo26a}.
\begin{proposition}\label{P2.2}
Let $\psi\in\pmb{\Psi}_n$.
The following assertions hold.
\begin{enumerate}[\rm (i)]
\item\label{P2.2-1}
$\max\{t_1,\ldots,t_n\}\le\psi(t)\le 1$ for all $t:=(t_1,\ldots,t_n)\in\Delta_n$.
\item\label{P2.2-3}
If $0\le\alpha_i\le\beta_i$ $(i=1,\ldots,n)$ with $\sum_{i=1}^{n}\al_i>0$, then 
\begin{equation*}
\left(\sum_{i=1}^{n}\alpha_i\right)\psi\left(\frac{\alpha_1}{\sum_{i=1}^{n}\alpha_i},\ldots,\frac{\alpha_n}{\sum_{i=1}^{n}\alpha_i}\right)
\le\left(\sum_{i=1}^{n}\be_i\right)\psi\left(\frac{\be_1}{\sum_{i=1}^{n}\be_i},\ldots,\frac{\be_n}{\sum_{i=1}^{n}\be_i}\right).
\end{equation*}
\item\label{P2.2-2}
If $\psi$ is {permutation-symmetric}, then $\psi$ attains its minimum at $\left(\frac{1}{n},\ldots,\frac{1}{n}\right)$.
\item\label{P2.2-5} 
The function defined for $x:=(x_1,\ldots,x_n)\in\R^n$ by
\begin{equation}\label{T2.2-1}
\|x\|_\psi
:= \begin{cases}
\left(\sum_{i=1}^n|x_i|\right)\cdot\psi\left(\dfrac{|x_1|}{\sum_{i=1}^n|x_i|},\ldots,\dfrac{|x_{n}|}{\sum_{i=1}^n|x_i|}\right)  & \text{\rm if}\;x\ne0_{\R^n},\\
0 & \text{\rm otherwise}
\end{cases} 
\end{equation}
is a norm on $\R^n$.
\item\label{P2.2-6}
For all $x:=(x_1,\ldots,x_n)\in\R^n$, the dual norm of \eqref{T2.2-1} is given by
\begin{gather*}
\|x\|^*_{\psi} =\max_{(t_1,\ldots,t_n)\in\Delta_n}\dfrac{\sum_{i=1}^nt_i|x_i|}{\psi (t_1,\ldots,t_{n})}.
\end{gather*}	
\item\label{P2.2-7}
Let $\psi^*:\Delta_n\to\R$ be given by 
\begin{gather}\label{psi*}
\psi^*(t):=\max_{s\in\Delta_n}\dfrac{\langle s,t\rangle}{\psi(s)}
\;\;
\text{\rm for all}
\;\; t\in\Delta_n.
\end{gather}	
Then $\psi^*\in\pmb{\Psi}_n$ and $\|\cdot\|^*_{\psi}=\|\cdot\|_{\psi^*}$.
\end{enumerate}
\end{proposition}	

\begin{remark}\label{R2}
For all $(t_1,\ldots,t_n)\in\Delta_n$, define
\begin{gather*}
\psi_p(t_1,\ldots,t_n)
:=\begin{cases}
\left(t_1^p+\cdots+t_n^p\right)^{\frac{1}{p}}  & \text{if}\;p\in[1,\infty),\\
\max\{t_1,\ldots,t_n\} & \text{if}\;p=\infty.
\end{cases} 
\end{gather*}
Then 
$\psi_p\in\pmb{\Psi}_n$, and 
\begin{equation*}
\|x\|_{\psi_p}
=\begin{cases}
\left(|x_1|^p+\cdots+|x_n|^p\right)^{\frac{1}{p}}  & \text{if}\;p\in[1,\infty),\\	
\max\{|x_1|,\ldots,|x_n|\} & \text{if}\; p=\infty
\end{cases} 
\end{equation*}
for all $x:=(x_1,\ldots,x_n)\in \R^n$. 
By \eqref{psi*}, $\psi^*_p=\psi_q$ with $q\in[1,\infty]$ satisfying $\frac{1}{p}+\frac{1}{q}=1$.
The reader is referred to \cite{Cuo26a,SaiKatTak00} and the references therein for further examples of norms generated by functions in $\pmb{\Psi}_n$.
\end{remark}

\section{A General Convex Composition Minimisation Problem}\label{S3}
Let $f_1,\ldots,f_n:X\to\R_+$ be
locally Lipschitz continuous and convex.
Define $f:X\to\R^n_+$ by 
\begin{gather}\label{f}
f(x):=(f_1(x),\ldots,f_n(x))\;\;\text{for all}\;\;x\in X.
\end{gather}
Let $\Phi:\R^n_+\to\R_+$ be a convex function satisfying
\begin{gather}\label{mt}
\Phi(u_1,\ldots,u_n)\le\Phi(v_1,\ldots,v_n)\;\;\text{for all}\;\;
0\le u_i\le v_i\;(i=1,\ldots,n).
\end{gather}	
Consider the minimisation problem:
\begin{gather}
\label{P}
\tag{$P$}
{\rm{minimise}}\;F(x):=(\Phi\circ f)(x)\;\;\text{subject to}\;\;x\in X.
\end{gather}
Denote by $\operatorname{Sol}(P)$ the solution set of problem \eqref{P}.

\begin{lemma}\label{L3}
Define $\widetilde\Phi:\R^n\to\R_+$ by
\begin{gather}\label{d}
\widetilde\Phi(u):=\Phi(u^+)\;\;\text{for all}\;\;u:=(u_1,\ldots,u_n)\in\R^n.
\end{gather}
Then $\widetilde\Phi$ is convex and coordinatewise nondecreasing on $\R^n$, i.e.,
\begin{gather*}
\widetilde\Phi(u_1,\ldots,u_n)\le\widetilde\Phi(v_1,\ldots,v_n)\;\;\text{for all}\;\;
u_i\le v_i\;(i=1,\ldots,n).
\end{gather*}	
Moreover, $\widetilde\Phi=\Phi$ on $\R^n_+$.
\end{lemma}	

\begin{proof}[\normalsize Proof]\normalsize
Let $u:=(u_1,\ldots,u_n),v:=(v_1,\ldots,v_n)\in\R^n$ and $\al\in[0,1]$.
We have 
\begin{gather}\label{c}
[\al u_i+(1-\al)v_i]^+\le\al u^+_i+(1-\al)v^+_i\;\;(i=1,\ldots,n).
\end{gather}
By \eqref{mt}--\eqref{c}, and the convexity of $\Phi$,
\begin{align*}
\widetilde\Phi(\al u+(1-\al)v)
&=\Phi((\al u+(1-\al)v)^+)\\
&\le\Phi(\al u^++(1-\al)v^+)\\
&\le\al\Phi(u^+)+(1-\al)\Phi(v^+)\\
&=\al\widetilde\Phi(u)+(1-\al)\widetilde\Phi(v).
\end{align*}
Thus, $\widetilde\Phi$ is convex.
If $u_i\le v_i$, then $0\le u^+_i\le v^+_i$  $(i=1,\ldots,n)$.
By \eqref{mt},
\begin{gather*}
\widetilde{\Phi}(u)=\Phi(u^+)\le\Phi(v^+)=\widetilde{\Phi}(v).
\end{gather*}	
Thus, $\widetilde\Phi$ is coordinatewise nondecreasing on $\R^n$.
The last assertion is obvious.
\end{proof}	

\begin{proposition}\label{P3.1}
\begin{enumerate}[\rm (i)]
\item\label{P3.1-1}		
The function $F$ is locally Lipschitz continuous and convex.	
Consequently, 
$\operatorname{Sol}(P)$ is closed and convex.	
\item\label{P3.1-3}
If $X$ is  a reflexive Banach space and $F$ is coercive, then  $\operatorname{Sol}(P)$ is nonempty and bounded.
\end{enumerate}
\end{proposition}	

\begin{proof}[\normalsize Proof]\normalsize
Let $\bx\in X$ and $\|\cdot\|_{\R^n}$ be an arbitrary norm on $\R^n$.
Then, there exists a $\xi>0$ such that $\|\cdot\|_{\R^n}\le \xi\|\cdot\|_{\psi_\infty}$.
Since $f_1,\ldots,f_n$ have nonnegative values, we have $F=\widetilde\Phi\circ f$.
By Lemma~\ref{L3} and \cite[Corollary~2.153]{MorNam22}, there exists a neighbourhood $V$ of $f(\bx)$ and a constant $l_{\widetilde\Phi}>0$ such that
\begin{gather*}
|\widetilde\Phi(u)-\widetilde\Phi(v)|\le l_{\widetilde\Phi}\|u-v\|_{\R^n}\;\;\text{for all}\;u,v\in V.
\end{gather*}	
By the local Lipschitz continuity of $f_1,\ldots,f_n$ and the  continuity of $f$, there exists a neighbourhood $U$ of $\bx$ and constants $l_{f_1},\ldots,l_{f_n}>0$ such that $f(U)\subset V$ and
\begin{gather*}
|f_i(x)-f_i(y)|\le l_{f_i}\|x-y\|\;\;\text{for all}\;x,y\in U\;(i=1,\ldots,n).
\end{gather*}	
For every $x,y\in U$, we have
\begin{align*}
|F(x)-F(y)|
&=|\widetilde\Phi(f(x))-\widetilde\Phi(f(y))|\\
&\le l_{\widetilde\Phi}\|f(x)-f(y)\|_{\R^n}\\
&\le l_{\widetilde\Phi}\xi\max_{1\le i\le n}|f_i(x)-f_i(y)|\\
&\le l_{\widetilde\Phi}\xi\max\{l_{f_1},\ldots,l_{f_n}\}\|x-y\|.
\end{align*}	
Thus, $F$ is locally Lipschitz continuous at $\bx$.

Let $x,y\in X$ and $\al\in[0,1]$.
Then
\begin{align*}
F(\al x+(1-\al)y)
&\le \Phi(\al f_1(x)+(1-\al)f_1(y),\ldots,\al f_n(x)+(1-\al)f_n(y))\\
&=\Phi(\al f(x)+(1-\al)f(y))\\
&\le \al\Phi(f(x))+(1-\al)\Phi(f(y))\\
&=\al F(x)+(1-\al)F(y).
\end{align*}	
Thus, $F$ is convex.
Its continuity and convexity imply that
$\operatorname{Sol}(P)$ is closed and convex.
 
Suppose that $X$ is a reflexive Banach space  and $F$ is coercive.
By \cite[Corollary~3.23]{Bre11}, 
$\text{\rm {Sol}}(P)$ is nonempty.
If $\text{\rm {Sol}}(P)$ is unbounded,
then there exists a sequence $\{x_k\}\subset\text{\rm {Sol}}(P)$ with $\|x_k\|\to\infty$.
By coercivity, $F(x_k)\to\infty$, which contradicts the fact that $F$ has the same finite value at every solution.
\end{proof}	

\begin{remark}\label{R3.3}
\begin{enumerate}[\rm (i)]
\item\label{R3.3-1}
Let  $\Phi(u):=\|u\|_\psi$ for all $u\in\R^n_+$ and some $\psi\in \pmb{\Psi}_n$.
By Proposition~\ref{P2.2}\eqref{P2.2-1}, $F$ is coercive if there exists an $i\in\{1,\ldots,n\}$ such that  $f_i$ is coercive.
By Proposition~\ref{P2.2}\eqref{P2.2-3},  $\|\cdot\|_\psi$ satisfies the  monotonicity property \eqref{mt}.
\item\label{R3.3-2}
If $\Phi$ and $f_1,\ldots,f_n$ are globally Lipschitz continuous, then so is $F$.
\end{enumerate}
\end{remark}	

\begin{proposition}\label{P3.3}
Let $\bx\in X$, and $f$ be given by \eqref{f}.
Then
\begin{gather*}
\partial F(\bx)=
\left\{\sum_{i=1}^{n}\al_ix^*_i\Bigm|
(\al_1,\ldots,\al_n)\in\partial\Phi(f(\bx)),\;
\al_i\ge0,\;
x_i^*\in\partial f_i(\bx)\;(i=1,\ldots,n)\right\}.
\end{gather*}
\end{proposition}

\begin{proof}[\normalsize Proof]\normalsize
Let $\widetilde{\Phi}$ be given by \eqref{d}.
By Lemma~\ref{L3}, it is convex and  coordinatewise nondecreasing on $\R^n$.
We first prove that
\begin{gather}\label{e}
\partial\widetilde{\Phi}(u)=\partial\Phi(u)\cap\R^n_+\;\;\text{for all}\;\;u\in\R^n_+.	
\end{gather}	
Let
$u\in\R^n_+$ and $\al:=(\al_1,\ldots,\al_n)\in\partial\widetilde{\Phi}(u)$.
For any $t>0$, we have
\begin{gather*}
-t\al_i=\langle\al,(u-t\mathbf{e}_i)-u\rangle
\le\widetilde\Phi(u-t\mathbf{e}_i)-\widetilde\Phi(u)
\le0\;\;(i=1,\ldots,n).
\end{gather*}
Hence, $\al_i\ge0$ $(i=1,\ldots,n)$.
Since $\widetilde\Phi=\Phi$ on $\R^n_+$, restricting the subgradient inequality to $\R^n_+$ gives $\al\in\partial\Phi(u)$.
Conversely, let $\al\in\partial\Phi(u)$ with
$\al_i\ge0$ $(i=1,\ldots,n)$.
For every $v\in\R^n$, we have
\begin{gather*}
\langle\al,v-u\rangle
\le\langle\al,v^+-u\rangle
\le\Phi(v^+)-\Phi(u)
=\widetilde\Phi(v)-\widetilde\Phi(u).
\end{gather*}
Thus, $\al\in\partial\widetilde\Phi(u)$.
This proves \eqref{e}.
Since $f_i$ $(i=1,\ldots,n)$ has non-negative values, we have $F=\widetilde\Phi\circ f$.
By Lemma~\ref{L1.2},
\begin{gather*}
\partial F(\bx)= \left\{\sum_{i=1}^{n}\al_ix^*_i\Bigm|(\al_1,\ldots,\al_n)\in\partial \widetilde\Phi(f(\bx)),\ x_i^*\in\partial f_i(\bx)\; (i=1,\ldots,n)\right\}.
\end{gather*}	
From this and \eqref{e}, we obtain the desired statement.
\end{proof}

\begin{theorem}\label{T3.6}
Let $\bx\in X$.
The following assertions hold.
\begin{enumerate}[\rm (i)]
\item\label{T3.6-1}
$\bar x\in \text{\rm Sol}(P)$ if and only if there exist
$(\al_1,\ldots,\al_n)\in\partial \Phi(f(\bx))$ with $\al_i\ge 0$ and $x_i^*\in\partial f_i(\bx)$ $(i=1,\ldots,n)$ such that
$\displaystyle\sum_{i=1}^{n}\al_ix^*_i=0_{X^*}$.
\item\label{T3.6-2}
Let $\bar x\in \text{\rm Sol}(P)$,  $\al_i$ and  $x_i^*$ $(i=1,\ldots,n)$ satisfy the conditions in assertion \eqref{T3.6-1}.
Then
\begin{gather}\label{T3.5-0}
\text{\rm Sol}(P)=
\left\{x\in X\;\middle|\;\;
\begin{aligned}
& x_i^*\in\partial f_i(x)\;\;\text{for every}\;i\;\text{with}\;\al_i>0,\\ 
& (\al_1,\ldots,\al_n)\in\partial\Phi(f(x))
\end{aligned}
\right\}.
\end{gather}
\end{enumerate}
\end{theorem}	

\begin{proof}[\normalsize Proof]\normalsize
Assertion \eqref{T3.6-1} follows from Proposition~\ref{P3.3}.
Denote by $\mathrm R$ the set on the right-hand side
of \eqref{T3.5-0} and $\al:=(\al_1,\ldots,\al_n)$.
Let $x\in\text{\rm R}$.
Then for every $i$ with $\al_i>0$, we have 
$x_i^*\in\partial f_i(x)\cap \partial f_i(\bx)$, and consequently,
$f_i(x)-f_i(\bx)=\langle x_i^*,x-\bx\rangle$.
Then
\begin{gather*}
\langle \al,f(x)-f(\bx)\rangle=\sum_{i=1}^n\al_i(f_i(x)-f_i(\bx))
=\sum_{i=1}^n\al_i\langle x_i^*,x-\bx\rangle
=\Big\langle \sum_{i=1}^n\al_ix_i^*,x-\bx\Big\rangle=0.
\end{gather*}	
Since $\al\in\partial\Phi(f(x))$, we have
\begin{gather*}
\Phi(f(\bx))-\Phi(f(x))\ge \langle \al,f(\bx)-f(x)\rangle=0.
\end{gather*}
On the other hand, we always have $\Phi(f(\bx))-\Phi(f(x))\le 0$.
Thus, $\Phi(f(\bx))=\Phi(f(x))$.
Therefore, $x\in\operatorname{Sol}(P)$.

Let $x\in\text{\rm Sol}(P)$. 
Then
\begin{gather*}
0=\Phi(f(x))-\Phi(f(\bx))\ge \langle \al,f(x)-f(\bx)\rangle.
\end{gather*}
We have $f_i(x)-f_i(\bx)\ge \langle x^*_i,x-\bx\rangle$ $(i=1,\ldots,n)$, and consequently,
\begin{gather*}
\langle \al,f(x)-f(\bx)\rangle\ge\sum_{i=1}^n\al_i\langle x_i^*,x-\bx\rangle
=\Big\langle \sum_{i=1}^n\al_ix_i^*,x-\bx\Big\rangle=0.
\end{gather*}	
Therefore, $\langle \al,f(x)-f(\bx)\rangle=0$, and consequently,
\begin{gather*}
\sum_{i=1}^{n}\al_i\left(f_i(x)-f_i(\bx)-\langle x^*_i,x-\bx\rangle\right)=0.
\end{gather*}
Thus, $f_i(x)-f_i(\bx)-\langle x^*_i,x-\bx\rangle=0$ when $\al_i>0$.
Fix  $i\in\{1,\ldots,n\}$ with $\al_i>0$.
Since $x_i^*\in\partial f_i(\bx)$, we have
\begin{align*}
\langle x_i^*,u-x\rangle
&=\langle x_i^*,u-\bx\rangle+\langle x_i^*,\bx-x\rangle\\
&\le (f_i(u)-f_i(\bx))+(f_i(\bx)-f_i(x))
=f_i(u)-f_i(x)
\end{align*}
for every $u\in X$.
Thus, $x_i^*\in \partial f_i(x)$ when $\al_i>0$.

Let $y\in \R^n_+$.
Since  $\al\in\partial\Phi(f(\bx))$ and $\Phi(f(\bx))=\Phi(f(x))$, we have
$\Phi(y)-\Phi(f(x))\ge \langle \al,y-f(\bx)\rangle$, and consequently,
\begin{align*}
\Phi(y)-\Phi(f(x))
&\ge\langle \al,y-f(\bx)\rangle\\
&=\langle\al,y-f(x)\rangle+\langle\al,f(x)-f(\bx)\rangle =\langle\al,y-f(x)\rangle.
\end{align*}
Thus, $\al\in\partial\Phi(f(x))$.
Hence, $x\in \text{\rm R}$.
This proves \eqref{T3.6-2}.
\end{proof}	

\begin{remark}
In \cite[Lemma 2(a) \& Theorem 5]{Bur91}, the authors showed that the set of subgradients satisfying the optimality condition is independent of the chosen solution. 
Theorem~\ref{T3.6}\eqref{T3.6-2} complements this result by providing a characterisation adapted to the composition structure $F=\Phi\circ f$.
\end{remark}

The next statement is a direct consequence of Theorem~\ref{T3.6}, formulas \eqref{sn} and \eqref{T2.2-1}.
\begin{corollary}\label{C3.5}
Let $\Phi(u):=\|u\|_\psi$ for all $u\in\R^n_+$ and some $\psi\in\pmb{\Psi}_n$, and $\bx\in X$.
The following assertions hold.
\begin{enumerate}[\rm (i)]
\item\label{C3.5-1}
$\bar x\in \text{\rm Sol}(P)$ if and only if there exist
$(\al_1,\ldots,\al_n)\in\R^n_+$ and $x_i^*\in\partial f_i(\bx)$ $(i=1,\ldots,n)$ such that
\begin{gather*}
\|(\al_1,\ldots,\al_n)\|^*_\psi\le 1,\;\; \sum_{i=1}^{n}\al_if_i(\bx)=\|(f_1(\bx),\ldots,f_n(\bx))\|_\psi,\;\;\sum_{i=1}^{n}\al_ix^*_i=0_{X^*}.
\end{gather*}	
If $f_i(\bx)>0$ for some $i\in\{1,\ldots,n\}$, then $\|(\al_1,\ldots,\al_n)\|^*_\psi=1$.
\item\label{C3.5-2}
Let $\bar x\in \text{\rm Sol}(P)$, $\al_i$ and $x_i^*$ $(i=1,\ldots,n)$ satisfy the conditions in assertion \eqref{C3.5-1}.
Then
\begin{gather*}
\text{\rm Sol}(P)=
\left\{x\in X\;\middle|\;\;
\begin{aligned}
& x_i^*\in\partial f_i(x)\;\;\text{for every}\;i\;\text{with}\;\al_i>0,\\ 
& \sum_{i=1}^{n}\al_if_i(x)=\|(f_1(x),\ldots,f_n(x))\|_\psi
\end{aligned}
\right\}.
\end{gather*}
\end{enumerate}
\end{corollary}	

\section{Optimisation with Minimal Time Functions}\label{S4}
Let $F_1,\ldots,F_n, \Omega_1,\ldots,\Omega_n \subset X$ be nonempty, closed and convex.
Suppose that $F_1,\ldots,F_n$ are bounded and  $0_X\in\bigcap_{i=1}^n\operatorname{int}F_i$.
In this section, we consider a particular case of problem \eqref{P} with $\Phi(u):=\|u\|_\psi$ for all $u\in\R^n_+$ and some $\psi\in\pmb{\Psi}_n$, 
\begin{gather*}
f_i(x):=\mathcal{T}^{F_i}_{\Omega_i}(x)\;\;\text{for all}\;\;x\in X\;(i=1,\ldots,n).
\end{gather*}
These functions are finite-valued, convex and globally Lipschitz continuous.
The optimisation model is of the form:
\begin{gather*}
{\rm{minimise}}\;F(x):=\|(\mathcal{T}^{F_1}_{\Omega_1}(x),\ldots,\mathcal{T}^{F_n}_{\Omega_n}(x))\|_\psi\;\;\text{subject to}\;\;x\in X.
\end{gather*}

\begin{proposition}\label{P4.3}
If $X$ is  a reflexive Banach space and  one of the sets $\Omega_1,\ldots,\Omega_n$ is bounded, then  $\operatorname{Sol}(P)$ is nonempty and bounded.
\end{proposition}	

\begin{proof}[\normalsize Proof]\normalsize
Without loss of generality, we can assume that $\Omega_1$ is bounded.
Since $F_1$ and $\Omega_1$ are bounded, there exists a $\gamma>0$ such that $F_1\subset \gamma\B_X$ and $\Omega_1\subset\gamma\B_X$.
Let $x\in X$. 
Then the set
$T(x):=\{t\in\R_+\mid\Omega_1\cap(x+tF_1)\ne\emptyset\}$ is nonempty. 
For each $t\in T(x)$, there exist $z\in\Omega_1$ and $u\in F_1$ such that $z=x+tu$. 
Thus,
$d(x,\Omega_1)\le\|x-z\|=t\|u\|\le t\gamma$, and consequently,
$f_1(x)=\inf T(x)\ge\frac{1}{\gamma}d(x,\Omega_1).$
We have
\begin{align*}
f_1(x)
\ge\frac{1}{\gamma}\left(\inf_{\omega\in\Omega_1}\|x-\omega\|\right)
\ge\frac{1}{\gamma}\left(\inf_{\omega\in\Omega_1}(\|x\|-\|\omega\|)\right)
\ge\frac{\|x\|}{\gamma}-1.
\end{align*}	
Thus,
$f_1(x)\to\infty\;\;\text{as}\;\;\|x\|\to\infty$.
By Remark~\ref{R3.3}\eqref{R3.3-1}, $F$ is coercive.
The statement is a consequence of Proposition~\ref{P3.1}\eqref{P3.1-3}.
\end{proof}	

\begin{remark}
\begin{enumerate}[\rm (i)]
\item		
The sum norm version of Proposition~\ref{P4.3} can be found in	\cite[Proposition~4.1]{MorNam11+}.
\item 
By Proposition~\ref{P3.1}\eqref{P3.1-1} and Remark~\ref{R3.3}\eqref{R3.3-2}, $F$ is finite-valued,  convex and globally Lipschitz continuous. 
Thus, $\operatorname{Sol}(P)$ is closed and convex.	
\end{enumerate}
\end{remark}

The next statement is a consequence of Corollary~\ref{C3.5}.
\begin{theorem}\label{T3.5}
Let $\bar{x}\in X$.
The following assertions hold.
\begin{enumerate}[\rm (i)]
\item\label{T3.5-1}
$\bar x\in \text{\rm Sol}(P)$ if and only if there exist $\al_i\ge 0$ and $x_i^*\in \partial \mathcal{T}^{F_i}_{\Omega_i}(\bx)$ $(i=1,\ldots,n)$ such that
\begin{gather*}
\|(\al_1,\ldots,\al_n)\|_{\psi}^*\le 1,\;
\sum_{i=1}^{n}\al_i\mathcal{T}^{F_i}_{\Omega_i}(\bx)=\|(\mathcal{T}^{F_1}_{\Omega_1}(\bx),\ldots,\mathcal{T}^{F_n}_{\Omega_n}(\bx))\|_{\psi},\;\sum_{i=1}^{n}\al_ix^*_i=0_{X^*}.
\end{gather*}	
If $\bx\notin\Omega_i$ for some $i\in\{1,\ldots,n\}$, then $\|(\al_1,\ldots,\al_n)\|^*_\psi=1$.
\item\label{T3.5-2}
Let $\bar x\in \text{\rm Sol}(P)$, $\al_i$ and $x_i^*$ $(i=1,\ldots,n)$ satisfy the conditions in assertion \eqref{T3.5-1}.
Then
\begin{gather*}
\text{\rm Sol}(P)=
\left\{x\in X\;\middle|\;\;
\begin{aligned}
& x_i^*\in\partial\mathcal{T}^{F_i}_{\Omega_i}(x)\;\;\text{for every}\;i\;\text{with}\;\al_i>0,\\ 
& \sum_{i=1}^n\al_i\mathcal{T}^{F_i}_{\Omega_i}(x)=\|(\mathcal{T}^{F_1}_{\Omega_1}(x),\ldots,\mathcal{T}^{F_n}_{\Omega_n}(x))\|_\psi
\end{aligned}
\right\}.
\end{gather*}
\end{enumerate}
\end{theorem}

\begin{remark}
\begin{enumerate}[\rm (i)]
\item		
When $F_1=\cdots=F_n=\B_X$, Theorem~\ref{T3.5} reduces to the corresponding result in \cite{HieCuo26}. 
In addition, if $\Omega_1,\ldots,\Omega_n$ are singletons, it recovers \cite[Theorem~4.6]{Cuo26b}.
\item 
If $\bx\in\bigcap_{i=1}^n\Omega_i$, then
$\mathcal{T}^{F_i}_{\Omega_i}(\bx)=0$ $(i=1,\ldots,n)$.
In this case,
\begin{gather*}
\operatorname{Sol}(P)=\{x\in X\mid F(x)=0\}
=\bigcap_{i=1}^n\Omega_i.
\end{gather*}
Then $\bx\in\operatorname{Sol}(P)$.
By the Fermat rule, $0_{X^*}\in\partial\mathcal{T}^{F_i}_{\Omega_i}(\bx)$ $(i=1,\ldots,n)$.
The conditions in  assertion \eqref{T3.5-1} hold
with $x_1^*=\cdots=x_n^*:=0_{X^*}$ and any
$(\al_1,\ldots,\al_n)\in\R^n_+$ satisfying
$\|(\al_1,\ldots,\al_n)\|_\psi^*\le1$.
The solution set formula in
assertion \eqref{T3.5-2} remains valid.
\end{enumerate}
\end{remark}

The following example illustrates  Theorem~\ref{T3.5}.
\begin{example}\label{E3.7}
Let $(X,\|\cdot\|)$ be a normed space, $\operatorname{dim}X\ge 2$, and $\mathbf{u}\in X$ be a unit vector.
Let $\gamma_1,\gamma_2\in\R$ with $\gamma_1<\gamma_2$, and $r>0$ with $\gamma_2-\gamma_1>2r$. 
Define
$x_1:=\gamma_1\mathbf{u}$, $x_2:=\gamma_2\mathbf{u}$, and $\bar x:=\frac{x_1+x_2}{2}$.
Then there exists a $\mathbf{u}^*\in X^*$
such that $\|\mathbf{u}^*\|^*=1$ and $\langle\mathbf{u}^*,\mathbf{u}\rangle=0$.
Define
\begin{gather*}
F_i:=\left\{x\in X\mid\phi_i(x)\le 1\right\}\;\;\text{and}\;\;	\Omega_i:=x_i+rF_i
\end{gather*}
with
$\phi_i(x):=\max\{\|x\|,i|\langle\mathbf{u}^*,x\rangle|\}$ for all $x\in X$ $(i=1,2)$.
Note that $\phi_1,\phi_2$ are norms on $X$.
The following assertions hold.
\begin{enumerate}[\rm (i)]
\item\label{E4.4.1}
$F_i$ and $\Omega_i$ $(i=1,2)$
are bounded, closed and convex  with $0_X\in\operatorname{int}F_1\cap\operatorname{int}F_2$ and $\bar{x}\notin\Omega_1\cup\Omega_2$.
Moreover,  $\Omega_i=\{x\in X\mid\phi_i(x-x_i)\le r\}$ and  $\rho_{F_i}=\phi_i$ $(i=1,2)$.
\item\label{E4.4.2}  $\mathcal{T}^{F_i}_{\Omega_i}(\cdot)=[\phi_i(\cdot-x_i)-r]^+$, 
\begin{gather}
\partial\mathcal{T}^{F_i}_{\Omega_i}(\bx)
=\bigg\{x_i^*\in X^*\mid
\sigma_{F_i}(x_i^*)=1,\;\langle x_i^*,\bx-x_i\rangle=\phi_i(\bx-x_i)\bigg\},
\label{E3.7-7}
\end{gather}	
and $P^{F_i}_{\Omega_i}(x)\ne\emptyset$  for every $x\in X$ $(i=1,2)$.
\item\label{E4.4.3}
There exists an $x^*_1\in X^*$ such that $\|x^*_1\|^*=1$ and $\langle x^*_1,\mathbf{u}\rangle=1$.
Set $x_2^*:=-x_1^*$. 
Then $x^*_i\in\partial\mathcal{T}^{F_i}_{\Omega_i}(\bx)$ $(i=1,2)$.
\item\label{E4.4.4}
Let $\psi\in \pmb{\Psi}_2$ be permutation-symmetric, and  $\al:=\psi\left(\frac{1}{2},\frac{1}{2}\right)>0$.
Then $\|(\al,\al)\|^*_{\psi}=1$, and
\begin{gather*}
\al\mathcal{T}^{F_1}_{\Omega_1}(\bx)+\al\mathcal{T}^{F_2}_{\Omega_2}(\bx)
=\|(\mathcal{T}^{F_1}_{\Omega_1}(\bx),\mathcal{T}^{F_2}_{\Omega_2}(\bx))\|_{\psi}.
\end{gather*}	
\item\label{E4.4.5}
$\bar x\in \text{\rm Sol}(P)$ and
\begin{gather*}
\text{\rm Sol}(P)=
\left\{x\in X\;\middle|\;\;
\begin{aligned}
&x_1^*\in\partial\mathcal{T}^{F_1}_{\Omega_1}(x),\; x_2^*\in\partial\mathcal{T}^{F_2}_{\Omega_2}(x),\\ 
&\al\mathcal{T}^{F_1}_{\Omega_1}(x)+\al\mathcal{T}^{F_2}_{\Omega_2}(x)
=\|(\mathcal{T}^{F_1}_{\Omega_1}(x),\mathcal{T}^{F_2}_{\Omega_2}(x))\|_{\psi}
\end{aligned}
\right\}.
\end{gather*}
\end{enumerate}
\end{example}	

\begin{proof}[\normalsize Proof]\normalsize
Since $\operatorname{dim}X\ge 2$, there exists a vector $\mathbf{v}\in X\setminus\text{span}\{\mathbf{u}\}$. 
Then the vectors $\mathbf{u}$ and $\mathbf{v}$ are linearly independent.
Let $M:=\text{span}\{\mathbf{u},\mathbf{v}\}$.
The linear functional  $l:M\to\R$ defined by
\begin{gather*}
l(\xi\mathbf{u}+\eta\mathbf{v}):=\eta\;\;\text{for all } \xi,\eta\in\R
\end{gather*}
is continuous since $M$ is finite-dimensional. 
By the Hahn-Banach theorem \cite[Corollary~1.2]{Bre11}, $l$ admits an extension $v^*\in X^*$.
Since  $\langle v^*,\mathbf{v}\rangle=1$, we have $v^*\ne 0_{X^*}$.
Note that $\langle v^*,\mathbf{u}\rangle=0$.
Set $\mathbf{u}^*:=\frac{v^*}{\|v^*\|^*}$.
Then $\|\mathbf{u}^*\|^*=1$ and $\langle\mathbf{u}^*,\mathbf{u}\rangle=0$.
\begin{enumerate}[\rm (i)]
\item 
Fix $i\in\{1,2\}$.
Since $\phi_i$ is a norm, $F_i$ is convex.
Observe that $\|x\|\le \phi_i(x)\le i\|x\|$ for all  $x\in X$.
Thus, $\phi_i$ is continuous with respect to the norm $\|\cdot\|$, and consequently, $F_i$ is closed.
One can check that $\frac{1}{i}
\B_X\subset F_i \subset \B_X$.
Hence, $F_i$ is bounded and 
$0_X\in\operatorname{int}F_i$.
As a consequence, $\Omega_i$ is  bounded, closed and convex.
Since $\Omega_i\subset x_i+r\B_X$ and 
$\|\bx-x_i\|=\frac{\gamma_2-\gamma_1}{2}>r$, we have
$\bx\notin\Omega_i$.
On the other hand,
\begin{gather*}
x\in\Omega_i\Leftrightarrow \frac{x-x_i}{r}\in F_i\Leftrightarrow \phi_i\left(\frac{x-x_i}{r}\right)\le 1\Leftrightarrow \phi_i(x-x_i)\le r.
\end{gather*}	
Thus, $\Omega_i=\{x\in X\mid\phi_i(x-x_i)\le r\}$.
If $x=0_X$, then $\rho_{F_i}(x)=\phi_i(x)=0$.
Suppose that $x\ne 0_X$.
Then
$\rho_{F_i}(x)=\inf\{t>0\mid x\in tF_i\}$.
For every $t>0$, it is clear that $x\in tF_i$ if and only if $\phi_i(x)\le t$.
Hence, $\rho_{F_i}(x)=\phi_i(x)$.
\item
Let $i=1,2$ and
$x\in X$.
If $x\in\Omega_i$, then $\phi_i(x-x_i)\le r$, and consequently,
\begin{gather*}
\mathcal{T}^{F_i}_{\Omega_i}(x)=0=[\phi_i(x-x_i)-r]^+.
\end{gather*}
Suppose that $x\notin\Omega_i$.
Then $\phi_i(x-x_i)>r$.
Define
\begin{gather*} 
z_i:=x_i+\frac{r}{\phi_i(x-x_i)}(x-x_i).
\end{gather*}	
Then $z_i\in\Omega_i$ since $\phi_i(z_i-x_i)=r$.
Since $z_i-x=-\left(1-\frac{r}{\phi_i(x-x_i)}\right)(x-x_i)$,
we have
\begin{gather}\label{E3.7-11}
\phi_i(z_i-x)=\phi_i(x-x_i)-r.
\end{gather}	
By \eqref{mi},
\begin{gather}
\mathcal{T}^{F_i}_{\Omega_i}(x)=\inf_{\omega\in\Omega_i}\rho_{F_i}(\omega-x)=\inf_{\omega\in\Omega_i}\phi_i(\omega-x).\label{E3.7-10}
\end{gather}	
Thus,
$\mathcal{T}^{F_i}_{\Omega_i}(x)
\le \phi_i(z_i-x)\overset{\eqref{E3.7-11}}{=}\phi_i(x-x_i)-r$.
On the other hand, for any $\omega\in\Omega_i$,  the triangle inequality gives
\begin{align*}
\phi_i(\omega-x)
=\phi_i(x-\omega)
\ge \phi_i(x-x_i)-\phi_i(\omega-x_i)
\ge\phi_i(x-x_i)-r.
\end{align*}
By \eqref{E3.7-10}, $\mathcal{T}^{F_i}_{\Omega_i}(x)
\ge\phi_i(x-x_i)-r.$
Thus, 
\begin{gather*}
\mathcal{T}^{F_i}_{\Omega_i}(x)=\phi_i(x-x_i)-r=[\phi_i(x-x_i)-r]^+.
\end{gather*}	
Since $\phi_i(\bx-x_i)-r>0$, by \cite[Theorem 3.59]{MorNam22}, 
\begin{gather}\label{E3.7-8}
\partial \mathcal{T}^{F_i}_{\Omega_i}(\bx)=\partial\phi_i(\cdot-x_i)(\bx)=\partial\rho_{F_i}(\cdot-x_i)(\bx).
\end{gather}
Since $\bx\ne x_i$, it follows from  \eqref{L2.2-1} and \eqref{E3.7-8} that
\begin{gather*}
\partial\mathcal{T}^{F_i}_{\Omega_i}(\bx)
=\left\{x_i^*\in X^*\mid\sigma_{F_i}(x_i^*)=1,\;\langle x_i^*,\bx-x_i\rangle=\rho_{F_i}(\bx-x_i)\right\}.
\end{gather*}
Hence, condition \eqref{E3.7-7} is satisfied.
If $x\in\Omega_i$, then $x\in P^{F_i}_{\Omega_i}(x)$.
Suppose that $x\notin\Omega_i$.
Then
\begin{gather*}
\phi_i(z_i-x)=\mathcal{T}^{F_i}_{\Omega_i}(x)\Leftrightarrow\phi_i\left(\frac{z_i-x}{\mathcal{T}^{F_i}_{\Omega_i}(x)}\right)=1.
\end{gather*} 
Hence, $z_i-x\in\mathcal{T}^{F_i}_{\Omega_i}(x)F_i$, and consequently,
$z_i\in (x+\mathcal{T}^{F_i}_{\Omega_i}(x)F_i)\cap\Omega_i$.
 By \eqref{mp}, $P^{F_i}_{\Omega_i}(x)\ne\emptyset$.
\item
By the Hahn-Banach theorem \cite[Corollary~1.130]{MorNam22},  there exists $x^*_1\in X^*$ such that $\|x^*_1\|^*=1$ and $\langle x^*_1,\mathbf{u}\rangle=1$.
Set $x_2^*:=-x_1^*$. 
Then $x^*_1+x^*_2=0_{X^*}$, $\|x^*_2\|^*=1$ and $\langle x^*_2,-\mathbf{u}\rangle=1$.
Since $\mathbf{u}\in F_1$, $-\mathbf{u}\in F_2$, and $F_i\subset\B_X$, we have $\sigma_{F_i}(x^*_i)=1$ $(i=1,2)$.
Let $\gamma:=\frac{\gamma_2-\gamma_1}{2}>r$. Then
$\bar{x}-x_1=\gamma\mathbf{u}$ and $\bar{x}-x_2=-\gamma\mathbf{u}$.
Thus, $\langle x_i^*,\bx-x_i\rangle=\gamma$ and
\begin{gather}\label{E4.4-13}
\phi_i(\bx-x_i)=\max\{\|\bx-x_i\|,i|\langle \mathbf{u}^*,\bx-x_i\rangle|\}=\max\{\gamma,0\}
=\gamma,
\end{gather}
and consequently, $\langle x_i^*,\bx-x_i\rangle=\phi_i(\bx-x_i)$ $(i=1,2)$.
By \eqref{E3.7-7},  $x^*_i\in \partial\mathcal{T}^{F_i}_{\Omega_i}(\bx)$ $(i=1,2)$.
\item
By Proposition~\ref{P2.2}\eqref{P2.2-2} and \eqref{P2.2-6},
\begin{equation*}
\|(\al,\al)\|^*_{\psi}=\max_{(t_1,t_2)\in\Delta_2}\frac{t_1\cdot\al+t_2\cdot\al}{\psi(t_1,t_2)}=\frac{\al}{\displaystyle\min_{(t_1,t_2)\in\Delta_2}\psi(t_1,t_2)}=1.
\end{equation*}
By assertion \eqref{E4.4.2},
$\mathcal{T}^{F_i}_{\Omega_i}(\bx)=\phi_i(\bx-x_i)-r\overset{\eqref{E4.4-13}}{=}\gamma-r>0$ $(i=1,2)$.
Thus,
\begin{align*}
\|(\mathcal{T}^{F_1}_{\Omega_1}(\bx),\mathcal{T}^{F_2}_{\Omega_2}(\bx))\|_{\psi}
&=\|(\gamma-r,\gamma-r)\|_\psi\\
&\overset{\eqref{T2.2-1}}{=}2(\gamma-r)\psi\left(\frac{1}{2},\frac{1}{2}\right)
=\al\mathcal{T}^{F_1}_{\Omega_1}(\bx)+\al\mathcal{T}^{F_2}_{\Omega_2}(\bx).
\end{align*}	
\item 
The statement is a  consequence of
Theorem~\ref{T3.5}.
\end{enumerate}
The proof is complete.
\end{proof}	

\section{Fermat–Torricelli Problem}\label{S5}
We study  the case
$\Phi(u):=\|u\|_{\psi_1}$ for all $u\in\R^n_+$.
Then
\begin{gather*}
F(x)=\mathcal{T}^{F_1}_{\Omega_1}(x)+\cdots+\mathcal{T}^{F_n}_{\Omega_n}(x)\;\;\text{for all}\;\;x\in X.
\end{gather*}	

\begin{theorem}\label{T2.7}
Let $\bar{x}\in X$.
The following assertions hold.
\begin{enumerate}[\rm (i)]
\item\label{T2.7.1}
$\bar x\in\text{\rm Sol}(P)$ if and only if there exist 
$x_i^*\in\partial\mathcal{T}^{F_i}_{\Omega_i}(\bx)$ $(i=1,\ldots,n)$ such that
\begin{gather}\label{T3.15-12}
\displaystyle\sum_{i=1}^{n}x^*_i=0_{X^*}.
\end{gather}
\item\label{T2.7.2}
Let $\bar x\in\text{\rm Sol}(P)$, and $x_1^*,\ldots,x^*_n$ satisfy the conditions in assertion \eqref{T2.7.1}.
Then 
\begin{gather}\label{SF}
\text{\rm Sol}(P)=\bigcap_{i=1}^n \text{\rm \textbf{A}}(F_i,\Omega_i,x_i^*)
\end{gather}
with
\begin{gather}\label{A+}
\text{\rm\textbf{A}}(F_i,\Omega_i,x^*_i):=
\left\{x\in X\mid x^*_i\in\partial \mathcal{T}^{F_i}_{\Omega_i}(x)\right\}\;\;(i=1,\ldots,n).
\end{gather}
\end{enumerate}
\end{theorem}

\begin{proof}[\normalsize Proof]\normalsize
By Remark~\ref{R2}, $\|\cdot\|^*_{\psi_1}=\|\cdot\|_{\psi_\infty}$.
By Theorem~\ref{T3.5}\eqref{T3.5-1}, $\bx\in\operatorname{Sol}(P)$ if and only if
there exist
$\al_i\ge0$ and $z_i^*\in\partial\mathcal{T}^{F_i}_{\Omega_i}(\bx)$
$(i=1,\ldots,n)$ such that
\begin{gather}\label{T3.5-10}
\sum_{i=1}^n\al_iz_i^*=0_{X^*},\;\;
\max_{1\le i\le n}\al_i\le1,\;\;
\sum_{i=1}^n\al_i\mathcal{T}^{F_i}_{\Omega_i}(\bx)
=\sum_{i=1}^n\mathcal{T}^{F_i}_{\Omega_i}(\bx).
\end{gather}
The last equality in  \eqref{T3.5-10} gives
\begin{gather}\label{T3.5-11}
\sum_{i=1}^n(1-\al_i)\mathcal{T}^{F_i}_{\Omega_i}(\bx)=0.
\end{gather}
Set $x^*_i:=\al_i z^*_i$ $(i=1,\ldots,n)$.
Then condition \eqref{T3.15-12} is satisfied.
By \eqref{T3.5-11},  $\al_i=1$ whenever $\bx\notin\Omega_i$ $(i=1,\ldots,n)$.
Fix $i\in\{1,\ldots,n\}$.
If $\bx\notin\Omega_i$, then  
$x^*_i=\al_iz_i^*=z_i^*\in
\partial\mathcal{T}^{F_i}_{\Omega_i}(\bx).$
Suppose that $\bx\in\Omega_i$. 
Then $\mathcal{T}^{F_i}_{\Omega_i}(\bx)=0$, and consequently,
$0_{X^*}\in\partial\mathcal{T}^{F_i}_{\Omega_i}(\bx)$.
Since this subdifferential is convex and $0\le\al_i\le1$, we have
\begin{gather*}
x^*_i=\al_iz_i^*=\al_iz_i^*+(1-\al_i)0_{X^*}\in\partial \mathcal{T}^{F_i}_{\Omega_i}(\bx).
\end{gather*}	
Thus, $x^*_i\in \partial\mathcal{T}^{F_i}_{\Omega_i}(\bx)$ $(i=1,\ldots,n)$.
Conversely, if there exist $x_i^*\in\partial\mathcal{T}^{F_i}_{\Omega_i}(\bx)$ $(i=1,\ldots,n)$ satisfying \eqref{T3.15-12}, then condition
\eqref{T3.5-10} holds true with $\al_i=1$ and $x^*_i$ in place of $z^*_i$ $(i=1,\ldots,n)$.
This proves assertion \eqref{T2.7.1}. 
Assertion \eqref{T2.7.2} is a consequence of 
Theorem~\ref{T3.5}\eqref{T3.5-2} with $\al_1=\cdots=\al_n=1$.
\end{proof}

The next statement decomposes the set  \eqref{A+} into the sum of two sets described by the normal cone to the target set and the Minkowski functional of the dynamics set.
\begin{proposition}\label{P3.7}
Let $F, \Omega\subset X$ be nonempty, closed and convex,  $x^*\in X^*$ with $\sigma_F(-x^*)\le 1$.
Suppose that $F$ is bounded, $0_X\in\operatorname{int}F$, and $P^F_\Omega(x)\ne \emptyset$ for any $x\in X$.
Then
\begin{gather}\label{P3.7-0}
\text{\rm\textbf{A}}(F,\Omega,x^*)=\text{\rm H}(\Omega,x^*)+\text{\rm T}(F,x^*)
\end{gather}
with
$\text{\rm H}(\Omega,x^*):=\{x\in \Omega\mid x^*\in N_\Omega(x)\}$ and
$\text{\rm T}(F,x^*):=\{z\in X\mid\langle x^*,z\rangle=\rho_F(-z)\}$.
\end{proposition}	

\begin{proof}[\normalsize Proof]\normalsize
Let $\bx\in \text{\rm\textbf{A}}(F,\Omega,x^*)$.
If $\bx\in\Omega$, then by the first formula in \eqref{sm}, $\bx\in \text{\rm H}(\Omega,x^*)$.
Since $0_X\in \text{\rm T}(F,x^*)$, we have
$\bx\in \text{\rm H}(\Omega,x^*)+\text{\rm T}(F,x^*)$.
Suppose  that $\bx\notin\Omega$.
By the second formula in \eqref{sm},
$x^*\in N_\Omega(x)$ and $\langle x^*,\bx-x\rangle=\rho_F(x-\bx)$ for every $x \in P^F_\Omega(\bx).$
Define $z:=\bx-x$ for some  $x \in P^F_\Omega(\bx)$. 
Then $z\in\text{\rm T}(F,x^*)$, and consequently, $\bx = x + z \in \text{\rm H}(\Omega,x^*)+\text{\rm T}(F,x^*)$.
	
Conversely, let 
$\bx := x + z$ for some $x \in \text{\rm H}(\Omega,x^*)$ and $z \in \text{\rm T}(F,x^*)$.
If $z=0_X$, then $\bx\in\Omega$ and $x^*\in N_\Omega(\bx)$.
Since $\sigma_F(-x^*)\le 1$, we have $x^*\in-F^\circ$.
By the first formula in \eqref{sm}, $\bx\in \text{\rm\textbf{A}}(F,\Omega,x^*)$.
Suppose that $z\ne 0_X$.
By Lemma~\ref{L2.7}\eqref{L2.7-1},
\begin{gather*}
0<\rho_{F}(-z)=\langle x^*,z\rangle	=\langle -x^*,-z\rangle	
\le\sigma_F(-x^*)\cdot\rho_F(-z)\le \rho_F(-z).
\end{gather*}	
Thus, $\sigma_F(-x^*)=1$.
Since $z\in\text{\rm T}(F,x^*)$, we have $\langle x^*,\bx-x\rangle=\rho_F(x-\bx)$.
By \eqref{L2.2-1}, $x^*\in-\partial\rho_F(x-\bx)$.
By Lemma~\ref{L2.7}\eqref{L2.7-2}, we have $\bx\notin\Omega$ and $x\in P^F_\Omega(\bx)$.
By the second formula in \eqref{sm}, $x^*\in\partial\mathcal{T}^F_\Omega(\bx)$, i.e.,
$\bx\in \text{\rm\textbf{A}}(F,\Omega,x^*)$.
\end{proof}

\begin{remark}\label{R4.6}
\begin{enumerate}[\rm (i)]
\item\label{R4.6-1}
When $F_1=\cdots=F_n=\B_X$, Theorem~\ref{T2.7} and Proposition~\ref{P3.7} recover the corresponding results in \cite{HieCuo26}. 
Moreover, if $\Omega_1,\ldots,\Omega_n$ are singletons, Theorem~\ref{T2.7} recaptures \cite[Theorem~5.1]{Cuo26b}.
\item 
When $F_1=\cdots=F_n$, assertion~\eqref{T2.7.1} recovers \cite[Theorem~4.3]{MorNam11+}. 
In addition, if they are singletons, it reduces to \cite[Theorem~5.4(iii)]{NamZal13} when the constraint set equals the whole space.
\end{enumerate}
\end{remark}	

The following examples illustrate Theorem~\ref{T2.7}.
\begin{example}\label{E4.4}
Let $X:=\R^2$ be equipped with the norm $\|\cdot\|:=\|\cdot\|_{\psi_p}$ with $p\in[1,\infty]$.
The dual norm is $\|\cdot\|^*=\|\cdot\|_{\psi_q}$ with $q\in[1,\infty]$ satisfying $\frac{1}{p}+\frac{1}{q}=1$.
Define 
\begin{gather}\label{E4.4-1}
\mathbf{u}:=(1,0),\;\mathbf{u}^*:=(0,1),\;
x_1:=-\mathbf{u},\;x_2:=3\mathbf{u},\;\bar x:=\mathbf{u},\\  \label{E4.4-2}
F_i:=\left\{(u_1,u_2)\mid \phi_i(u_1,u_2)\le 1\right\},\;
\Omega_i:=x_i+F_i,\\ \label{E4.4-0}
\phi_i(u_1,u_2)=\max\{\|(u_1,u_2)\|,i|u_2|\}\;\;\text{for all}\;\;(u_1,u_2)\in\R^2\;\;(i=1,2).
\end{gather}	
Then $\|\mathbf{u}\|=\|\mathbf{u}^*\|^*=1$ and $\langle \mathbf{u}^*,\mathbf{u}\rangle=0$.
The sets $F_i, \Omega_i$ $(i=1,2)$ are compact and convex with $0_{\R^2}\in\operatorname{int}F_1\cap\operatorname{int}F_2$ and $\bar{x}\notin\Omega_1\cup\Omega_2$.
By Example~\ref{E3.7}\eqref{E4.4.1}, $\rho_{F_i}=\phi_i$ $(i=1,2)$.
We have
\begin{gather}\label{E4.4-12}
F(x)=\mathcal{T}^{F_1}_{\Omega_1}(x)+\mathcal{T}^{F_2}_{\Omega_2}(x)
\end{gather}
for all $x\in\R^2.$ 
By Example~\ref{E3.7}\eqref{E4.4.5}, $\bar x\in\text{\rm Sol}(P)$. 
Define $x_1^*:=\mathbf{u}$ and $x_2^*:=-\mathbf{u}$.
Then $x^*_1+x^*_2=0_{\R^2}$,
$\|x^*_1\|^*=1$ and $\langle x^*_1,\mathbf{u}\rangle=1$.
By Example~\ref{E3.7}\eqref{E4.4.3}, 
$x_i^*\in\partial\mathcal{T}^{F_i}_{\Omega_i}(\bx)$ $(i=1,2)$.
\begin{enumerate}[\rm (i)]
\item\label{E4.4+1}
Let $p:=\infty$.
Then
\begin{gather}
F_1=[-1,1]\times[-1,1],\;\;\Omega_1=[-2,0]\times[-1,1],\label{E4.4-3}\\
F_2=[-1,1]\times[-1/2,1/2],\;\;\Omega_2=[2,4]\times[-1/2,1/2],
\label{E4.4-4}\\
\text{\rm H}(\Omega_1,x_1^*)=\{0\}\times[-1,1],\;\;\text{\rm H}(\Omega_2,x_2^*)=\{2\}\times[-1/2,1/2].\label{E4.4-5}
\end{gather}	

\begin{figure}[H]
\centering
\begin{tikzpicture}[scale=0.9]
\draw[gray!50,dashed] (-1.8,-1.8) grid (1.8,1.8);
		
\foreach \x in {0,1}
{
\draw (\x,0.05)--(\x,-0.05) node[anchor=north,xshift=4pt] at (\x,0) {\small $\x$};
}
		
\foreach \x in {-1}
{
\draw (\x,0.05)--(\x,-0.05) node[anchor=north east,xshift=4pt] at (\x,0) {\small $\x$};
}
\foreach \y in {-1}
\draw (0.05,\y)--(-0.05,\y) node[anchor=north east,xshift=4pt] {\small $\y$};
\foreach \y in {1}
\draw (0.05,\y)--(-0.05,\y) node[anchor=south east,xshift=4pt] {\small $\y$};

\draw[-{Stealth}] (-1.8,0)--(1.8,0) node[below right] {$x$};
\draw[-{Stealth}](0,-1.8)--(0,1.8) node[above left] {$y$};
		
\fill[red, fill opacity=0.4] (-1,-1) rectangle (1,1);
\node[red, above] at (0.5,1) {$F_1$};
\draw[red, thick] 	
(-1,-1) rectangle (1,1);
		
\fill[blue, fill opacity=0.4] (-1,-0.5) rectangle (1,0.5);
\node[blue, above] at (0.5,0.4) {$F_2$};
\draw[blue, thick] 	
(-1,-0.5) rectangle (1,0.5);

\begin{scope}[shift={(5.5,0)}]
\draw[gray!50,dashed] (-2.8,-1.8) grid (4.8,1.8);

\foreach \x in {0,1,2,3,4}
{
\draw (\x,0.05)--(\x,-0.05) node[anchor=north,xshift=4pt] at (\x,0) {\small $\x$};
}

\foreach \x in {-1}
{
\draw (\x,0.05)--(\x,-0.05) node[anchor=north east,xshift=4pt] at (\x,0) {\small $\x$};
}
\foreach \y in {-1}
\draw (0.05,\y)--(-0.05,\y) node[anchor=north east,xshift=4pt] {\small $\y$};
\foreach \y in {1}
\draw (0.05,\y)--(-0.05,\y) node[anchor=south east,xshift=4pt] {\small $\y$};

\draw[-{Stealth}] (-2.8,0)--(4.8,0) node[below right] {$x$};
\draw[-{Stealth}](0,-1.8)--(0,1.8) node[above left] {$y$};

\fill[red, fill opacity=0.4] (-2,-1) rectangle (0,1);
\node[red, above] at (-1,1) {$\Omega_1$};
\draw[red, thick] 	
(-2,-1) rectangle (0,1);

\fill[blue, fill opacity=0.4] (2,-0.5) rectangle (4,0.5);
\node[blue, above] at (3,0.4) {$\Omega_2$};
\draw[blue, thick] 	
(2,-0.5) rectangle (4,0.5);

\draw[red, line width=2pt] 
(0,1.02)--(0,-1.02);
\node[red, font=\small] at (0.8,0.5) {$\text{\rm H}(\Omega_1,x_1^*)$};

\draw[blue, line width=2pt] 	
(2,0.52)--(2,-0.52);
\node[blue, font=\small] at (2,-0.8) {$\text{\rm H}(\Omega_2,x_2^*)$};
\end{scope}
\end{tikzpicture}
\caption
{Example \ref{E4.4}\eqref{E4.4+1}: Sets \eqref{E4.4-3}--\eqref{E4.4-5}}
\label{fig1}
\end{figure}
\noindent
We have
\begin{gather}
\text{\rm T}(F_1,x_1^*)=\{(u_1,u_2)\mid u_1\ge|u_2|\},\;\text{\rm T}(F_2,x_2^*)=\{(u_1,u_2)\mid u_1\le -2|u_2|\}.\label{E4.4-6}
\end{gather}	
By \eqref{P3.7-0},
\begin{gather}
\text{\rm \textbf{A}}(F_1,\Omega_1,x_1^*)
=\{(u_1,u_2)\mid u_1\ge 0,\;|u_2|\le 1+u_1\},\label{E4.4-8}\\
\text{\rm \textbf{A}}(F_2,\Omega_2,x_2^*)
=\{(u_1,u_2)\mid u_1\le 2,\;2|u_2|\le 3-u_1\}.\label{E4.4-9}
\end{gather}	
By \eqref{SF}, 
\begin{gather}\label{E4.4-10}
\text{\rm Sol}(P)
=\left\{(u_1,u_2)\mid 0\le u_1\le 2,\;|u_2|\le\min\left\{u_1+1,\frac{3-u_1}{2}\right\}\right\},
\end{gather}
which is the hexagon with vertices 
\begin{gather*}
(0,1),\;\left(\frac{1}{3},\frac{4}{3}\right),\;\left(2,\frac{1}{2}\right),\;\left(2,-\frac{1}{2}\right),\;\left(\frac{1}{3},-\frac{4}{3}\right),\;(0,-1).
\end{gather*}	
\begin{figure}[H]
\centering
\begin{tikzpicture}[scale=0.9]
\begin{scope}
\draw[step=1cm,gray!50,dashed] (-3.5,-2.5) grid (2.5,2.5);
			
\foreach \x in {1,2}{
\draw (\x,0.05)--(\x,-0.05) node[anchor=north,xshift=4pt] at (\x,0) {\small $\x$};}

\foreach \x in {-3,-2,-1}
{\draw (\x,0.05)--(\x,-0.05) node[anchor=north east,xshift=4pt] at (\x,0) {\small $\x$};}

\foreach \y in {-2,-1,1,2}
\draw (0.05,\y)--(-0.05,\y) node[left] {\small $\y$};

\node[below] at (0.15,-0.1) {$0$};	
\draw[-{Stealth}](-3.5,0)--(2.5,0) node[below right] {$x$};
\draw[-{Stealth}](0,-2.5)--(0,2.5) node[above left] {$y$};
			
\clip (-3.5,-2.5) rectangle (2.5,2.5);
			
\fill[red, fill opacity=0.4]
(4,4)--(0,0)--(4,-4);
\draw[red, thick] (4,4)--(0,0)--(4,-4);
\node[red, font=\small] at (1.5,0.5) {$\text{\rm T}(F_1,x_1^*)$};
			
\fill[blue, fill opacity=0.4]
(-6,3)--(0,0)--(-6,-3);
\draw[blue, thick] (-6,3)--(0,0)--(-6,-3);
\node[blue, font=\small] at (-2.5,0.5) {$\text{\rm T}(F_2,x_2^*)$};
			
\fill(0,0) circle (1pt);
\end{scope}

\begin{scope}[shift={(6,0)}]
\draw[step=1cm,gray!50,dashed] (-2.5,-2.5) grid (3.5,2.5);

\foreach \x in {1,2,3}{
\draw (\x,0.05)--(\x,-0.05) node[anchor=north,xshift=4pt] at (\x,0) {\small $\x$};}

\foreach \x in {-2,-1}
{\draw (\x,0.05)--(\x,-0.05) node[anchor=north east,xshift=4pt] at (\x,0) {\small $\x$};}

\foreach \x in {0}
{\draw (\x,0.05)--(\x,-0.05) node[anchor=north east,xshift=3pt] at (\x,0) {\small $\x$};}

\foreach \y in {-2,-1,1,2}
\draw (0.05,\y)--(-0.05,\y) node[left] {\small $\y$};
\draw[-{Stealth}](-2.5,0)--(3.5,0) node[below right] {$x$};
\draw[-{Stealth}](0,-2.5)--(0,2.5) node[above left] {$y$};

\clip (-2.5,-2.5) rectangle (3.5,2.5);

\fill[red, fill opacity=0.4]
(4,3)--(2,3)--(0,1)--(0,-1)--(2,-3)--(4,-3);
\draw[red, thick] (4,3)--(2,3)--(0,1)--(0,-1)--(2,-3)--(4,-3);
\node[red, font=\small] at (2,1.4) {$\text{\rm \textbf{A}}(F_1,\Omega_1,x_1^*)$};

\fill[blue, fill opacity=0.4]
(-3,3)--(2,0.5)--(2,-0.5)--(-3,-3);
\draw[blue, thick] (-3,3)--(2,0.5)--(2,-0.5)--(-3,-3);
\node[blue, font=\small] at (-1.4,1.4) {$\text{\rm \textbf{A}}(F_2,\Omega_2,x_2^*)$};

\fill(0,1) circle (0.5pt);
\fill(1/3,4/3) circle (0.5pt);
\fill(2,1/2) circle (0.5pt);
\fill(2,-1/2) circle (0.5pt);
\fill(1/3,-4/3) circle (0.5pt);
\fill(0,-1) circle (0.5pt);
\node[violet] at (1,0.5) {$\text{\rm\textbf{Sol}}(\textbf{\textit{P}})$};
\end{scope}
\end{tikzpicture}
\caption{Example \ref{E4.4}\eqref{E4.4+1}: Sets \eqref{E4.4-6}--\eqref{E4.4-10}
}
\label{fig2}
\end{figure}
\item\label{E4.4+2}
Let $p\in[1,\infty)$.
Then
\begin{gather*}
F_1=\{(u_1,u_2)\mid|u_1|^p+|u_2|^p\le 1\},\;\text{\rm T}(F_1,x_1^*)=\{(u,0)\mid u\ge 0\},\\
F_2=\{(u_1,u_2)\mid |u_1|^p+|u_2|^p\le 1,\;|u_2|\le 1/2\},\;\text{\rm T}(F_2,x_2^*)=\{(u,0)\mid u\le 0\}.
\end{gather*}	
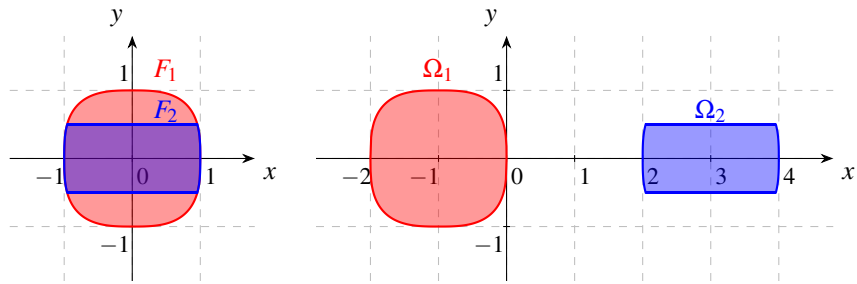
\begin{figure}[H]
\centering
\begin{tikzpicture}[scale=0.9]
\draw[gray!50,dashed] (-1.8,-1.8) grid (1.8,1.8);
		
\foreach \x in {0,1}
{\draw (\x,0.05)--(\x,-0.05) node[anchor=north,xshift=4pt] at (\x,0) {\small $\x$};}
		
\foreach \x in {-1}
{\draw (\x,0.05)--(\x,-0.05) node[anchor=north east,xshift=4pt] at (\x,0) {\small $\x$};}
\foreach \y in {-1}
\draw (0.05,\y)--(-0.05,\y) node[anchor=north east,xshift=4pt] {\small $\y$};
\foreach \y in {1}
\draw (0.05,\y)--(-0.05,\y) node[anchor=south east,xshift=4pt] {\small $\y$};
		
\draw[-{Stealth}] (-1.8,0)--(1.8,0) node[below right] {$x$};
\draw[-{Stealth}](0,-1.8)--(0,1.8) node[above left] {$y$};
		
\node[red, above] at (0.5,1) {$F_1$};
		
\fill[red, fill opacity=0.4]
plot[domain=-1:1,samples=200]
(\x,{(1-abs(\x)^3)^(1/3)})--
plot[domain=1:-1,samples=200]
(\x,{-(1-abs(\x)^3)^(1/3)})--cycle;
		
\draw[red, thick]
plot[domain=-1:1,samples=200]
(\x,{(1-abs(\x)^3)^(1/3)})--
plot[domain=1:-1,samples=200]
(\x,{-(1-abs(\x)^3)^(1/3)})--cycle;
		
\node[blue, above] at (0.5,0.4) {$F_2$};
		
\fill[blue, fill opacity=0.4]
plot[domain=-0.5:0.5,samples=200]
({(1-abs(\x)^3)^(1/3)},\x)--
plot[domain=0.5:-0.5,samples=200]
({-(1-abs(\x)^3)^(1/3)},\x);
		
\draw[blue, thick]
(-{(7/8)^(1/3)},0.5)--({(7/8)^(1/3)},0.5)--
plot[domain=0.5:-0.5,samples=200]
({-(1-abs(\x)^3)^(1/3)},\x)--
({(7/8)^(1/3)},-0.5)--(-{(7/8)^(1/3)},-0.5)--
plot[domain=-0.5:0.5,samples=200]
({(1-abs(\x)^3)^(1/3)},\x)--cycle;
		
\begin{scope}[shift={(5.5,0)}]
\draw[gray!50,dashed] (-2.8,-1.8) grid (4.8,1.8);
			
\foreach \x in {0,1,2,3,4}
{\draw (\x,0.05)--(\x,-0.05) node[anchor=north,xshift=4pt] at (\x,0) {\small $\x$};}
			
\foreach \x in {-2,-1}
{
\draw (\x,0.05)--(\x,-0.05) node[anchor=north east,xshift=4pt] at (\x,0) {\small $\x$};}
			
\foreach \y in {-1}
\draw (0.05,\y)--(-0.05,\y) node[anchor=north east,xshift=4pt] {\small $\y$};
\foreach \y in {1}
\draw (0.05,\y)--(-0.05,\y) node[anchor=south east,xshift=4pt] {\small $\y$};
			
\draw[-{Stealth}] (-2.8,0)--(4.8,0) node[below right] {$x$};
\draw[-{Stealth}](0,-1.8)--(0,1.8) node[above left] {$y$};
			
\node[red, above] at (-1,1) {$\Omega_1$};
\begin{scope}[shift={(-1,0)}]
\fill[red, fill opacity=0.4]
plot[domain=-1:1,samples=200]
(\x,{(1-abs(\x)^3)^(1/3)})--
plot[domain=1:-1,samples=200]
(\x,{-(1-abs(\x)^3)^(1/3)})--cycle;
				
\draw[red, thick]
plot[domain=-1:1,samples=200]
(\x,{(1-abs(\x)^3)^(1/3)})--
plot[domain=1:-1,samples=200]
(\x,{-(1-abs(\x)^3)^(1/3)})--cycle;
\end{scope}

\node[blue, above] at (3,0.4) {$\Omega_2$};
\begin{scope}[shift={(3,0)}]
\fill[blue, fill opacity=0.4]
plot[domain=-0.5:0.5,samples=200]
({(1-abs(\x)^3)^(1/3)},\x)--
plot[domain=0.5:-0.5,samples=200]
({-(1-abs(\x)^3)^(1/3)},\x);
				
\draw[blue, thick]
(-{(7/8)^(1/3)},0.5)--({(7/8)^(1/3)},0.5)--
plot[domain=0.5:-0.5,samples=200]
({-(1-abs(\x)^3)^(1/3)},\x)--
({(7/8)^(1/3)},-0.5)--(-{(7/8)^(1/3)},-0.5)--
plot[domain=-0.5:0.5,samples=200]
({(1-abs(\x)^3)^(1/3)},\x)--cycle;
\end{scope}
\end{scope}
\end{tikzpicture}
\caption{Example \ref{E4.4}\eqref{E4.4+2}: Sets $F_1, F_2$ and $\Omega_1, \Omega_2$ with $p=3$}
\end{figure}
\noindent
We have
\begin{gather*}
\Omega_1=\{(u_1,u_2)\mid |u_1+1|^p+|u_2|^p\le 1\},\;\text{\rm H}(\Omega_1,x_1^*)=\{(0,0)\},\\
\Omega_2=\{(u_1,u_2)\mid|u_1-3|^p+|u_2|^p\le 1,\;|u_2|\le 1/2\},\;\text{\rm H}(\Omega_2,x_2^*)=\{(2,0)\}.
\end{gather*}
\begin{figure}[H]
\centering
\begin{tikzpicture}[scale=0.9]
\draw[step=1cm,gray!50,dashed] (-2.5,-1.8) grid (4.5,1.8);

\foreach \x in {-1,0,1,2}{
\draw (\x,0.05)--(\x,-0.05) node[anchor=north,xshift=4pt] at (\x,0) {\small $\x$};}
\foreach \y in {-1}
\draw (0.05,\y)--(-0.05,\y) node[anchor=north east,xshift=4pt] {\small $\y$};
\foreach \y in {1}
\draw (0.05,\y)--(-0.05,\y) node[anchor=south east,xshift=4pt] {\small $\y$};
\draw[-{Stealth}](-2.5,0)--(4.5,0) node[below right] {$x$};
\draw[-{Stealth}](0,-1.8)--(0,1.8) node[above left] {$y$};

\draw[-{Stealth}, red, very thick] (0,0)--(4.5,0);
\node[red, font=\small] at (3.5,-0.4) {$\text{\rm T}(F_1,x_1^*)$};

\draw[blue, very thick] (-2.5,0)--(0,0);
\node[blue, font=\small] at (-1.8,-0.4) {$\text{\rm T}(F_2,x_2^*)$};

\fill[red](0,0) circle (3pt);
\node[red, font=\small] at (0.8,0.3) {$\text{\rm H}(\Omega_1,x_1^*)$};

\fill[blue](2,0) circle (3pt);
\node[blue, font=\small] at (2.8,0.3) {$\text{\rm H}(\Omega_2,x_2^*)$};

\begin{scope}[shift={(6.8,0)}]
\draw[step=1cm,gray!50,dashed] (-1.5,-1.8) grid (3.5,1.8);

\foreach \x in {-1,0,2}{
\draw (\x,0.05)--(\x,-0.05) node[anchor=north,xshift=4pt] at (\x,0) {\small $\x$};}
\foreach \y in {-1}
\draw (0.05,\y)--(-0.05,\y) node[anchor=north east,xshift=4pt] {\small $\y$};
\foreach \y in {1}
\draw (0.05,\y)--(-0.05,\y) node[anchor=south east,xshift=4pt] {\small $\y$};
\draw[-{Stealth}](-1.5,0)--(3.5,0) node[below right] {$x$};
\draw[-{Stealth}](0,-1.8)--(0,1.8) node[above left] {$y$};

\draw[-{Stealth}, red, very thick] (0,0)--(3.5,0);
\node[red, font=\small] at (2.4,0.3)
{$\text{\rm\textbf{A}}(F_1,\Omega_1,x_1^*)$};

\draw[blue, very thick] (-1.5,0)--(2,0);
\node[blue, font=\small] at (-0.5,0.3) {$\text{\rm \textbf{A}}(F_2,\Omega_2,x_2^*)$};

\draw[violet, very thick] (0,0)--(2,0);
\node[violet, font=\small] at (1,-0.4) {$\text{\rm Sol}(P)$};
\fill(0,0) circle (1pt);
\fill(2,0) circle (1pt);
\end{scope}
\end{tikzpicture}
\caption{Example \ref{E4.4}\eqref{E4.4+2}}
\end{figure}
\noindent
By \eqref{P3.7-0},
\begin{gather}\label{E4.4-11}
\text{\rm \textbf{A}}(F_1,\Omega_1,x_1^*)=\{(u,0)\mid u\ge 0\},\;\;\text{\rm \textbf{A}}(F_2,\Omega_2,x_2^*)=\{(u,0)\mid u\le 2\}.
\end{gather}
By \eqref{SF}, 
\begin{gather*}
\text{Sol}(P)=\text{\rm \textbf{A}}(F_1,\Omega_1,x_1^*)\cap\text{\rm \textbf{A}}(F_2,\Omega_2,x_2^*)=\{(u,0)\mid 0\le u\le 2\}.
\end{gather*}
\end{enumerate}
\end{example}

\begin{example}\label{E4.9}
Let $X:=L^2[0,1]$.
It is a Hilbert space with the inner product
\begin{gather*}
\langle x,y\rangle=\displaystyle\int_{0}^{1}x(t)y(t)dt\;\;\text{for all}\;\;x,y\in X.	
\end{gather*}	
Define 
\begin{gather}\label{E4.9-1}
\mathbf{u}(t):=\sqrt{3}t,\;\mathbf{u}^*(t):=2-3t\;\;\text{for all}\;\;t\in[0,1],\\ \label{E4.9-3}
x_1(t):=-t\;\;\text{for all}\;\;t\in[0,1],\;x_2:=-x_1,\;\bar x:=0_X,\\  \label{E4.9-2}
F_i:=\left\{z\in X\mid \phi_i(z)\le 1\right\},\;\Omega_i:=x_i+\frac{1}{2\sqrt 3}F_i,\\
\label{E4.9-4}
\phi_i(x)=\max\{\|x\|,i|\langle\mathbf{u}^*,x\rangle|\}\;\;\text{for all}\;\; x\in X\;\;(i=1,2).
\end{gather}	
Then $\|\mathbf{u}\|=\|\mathbf{u}^*\|^*=1$ and $\langle \mathbf{u}^*,\mathbf{u}\rangle=0$.
The sets  $\Omega_i,F_i$ $(i=1,2)$ are bounded, closed and convex with $0_{X}\in\operatorname{int}F_1\cap\operatorname{int}F_2$ and $\bar{x}\notin\Omega_1\cup\Omega_2$.
The objective function is of the form \eqref{E4.4-12} for all $x\in X$.
By Example~\ref{E3.7}\eqref{E4.4.5}, $\bar x\in\text{\rm Sol}(P)$. 
Let $x^*_1:=\mathbf{u}$ and $x^*_2:=-\mathbf{u}$.
Then $x^*_1+x^*_2=0_X$,
$\|x^*_1\|^*=1$ and $\langle x^*_1,\mathbf{u}\rangle=1$.
By Example~\ref{E3.7}\eqref{E4.4.3}, 
$x_i^*\in\partial \mathcal{T}^{F_i}_{\Omega_i}(\bx)$ $(i=1,2)$.
We have
\begin{gather*}
\mathrm H(\Omega_1,x_1^*)=\left\{-\frac{\mathbf{u}}{2\sqrt 3}\right\},\;\mathrm H(\Omega_2,x_2^*)=\left\{\frac{\mathbf{u}}{2\sqrt 3}\right\},\\	\mathrm T(F_1,x_1^*)=\{\lambda \mathbf{u}\mid \lambda\ge0\},\;
\mathrm T(F_2,x_2^*)=\{\lambda \mathbf{u}\mid \lambda\le0\}.
\end{gather*}	
By \eqref{P3.7-0},
\begin{gather}\label{E4.9-5}
\mathbf A(F_1,\Omega_1,x_1^*)
=\left\{\lambda \mathbf{u}\mid \lambda\ge-\frac{1}{2\sqrt 3}\right\}	
,\;
\mathbf A(F_2,\Omega_2,x_2^*)
=\left\{\lambda \mathbf{u}\mid \lambda\le\frac{1}{2\sqrt 3}\right\}.
\end{gather}
By \eqref{SF}, 
\begin{gather*}	
\text{Sol}(P)=\text{\rm \textbf{A}}(F_1,\Omega_1,x_1^*)\cap\text{\rm \textbf{A}}(F_2,\Omega_2,x_2^*)=\left\{\lambda\mathbf{u}\mid-\frac{1}{2\sqrt 3}\le\lambda\le\frac{1}{2\sqrt 3}\right\}.
\end{gather*}
\end{example}

\section{Chebyshev Centre Problem}\label{S6}
We study  the case
$\Phi(u):=\|u\|_{\psi_\infty}$ for all $u\in\R^n_+$.
Then
\begin{gather*}
F(x)=\max\left\{\mathcal{T}^{F_1}_{\Omega_1}(x),\ldots,\mathcal{T}^{F_n}_{\Omega_n}(x)\right\}\;\;\text{for all}\;\;x\in X.
\end{gather*}	
\begin{theorem}\label{T5.1}
Let $\bar{x}\in X$.
The following assertions hold.
\begin{enumerate}[\rm (i)]
\item\label{T5.1-1}
$\bar x\in \text{\rm Sol}(P)$ if and only if there exist $\al_i\ge 0$ and $x_i^*\in \partial\mathcal{T}^{F_i}_{\Omega_i}(\bx)$ $(i=1,\ldots,n)$ such that
\begin{gather}\label{T5.1-6}
\sum_{i=1}^{n}\al_ix^*_i=0_{X^*},\; 
\sum_{i=1}^{n}\al_i=1,\;
\left(\mathcal{T}^{F_i}_{\Omega_i}(\bx)-\max_{1\le k\le n}\mathcal{T}^{F_k}_{\Omega_k}(\bx)\right)\al_i=0\;\;(i=1,\ldots,n).
\end{gather}	
\item\label{T5.1-2}
Let $\bar x\in \text{\rm Sol}(P)$, $\al_i$ and $x_i^*$ $(i=1,\ldots,n)$ satisfy the conditions in assertion \eqref{T5.1-1}.
Then
\begin{gather}\label{SCC}
\text{\rm Sol}(P)=\bigcap_{\al_i>0}\text{\rm\textbf{B}}(F_i,\Omega_i,x_i^*),
\end{gather} 
where $\text{\rm \textbf{B}}(F_i,\Omega_i,x_i^*):=\text{\rm \textbf{A}}(F_i,\Omega_i,x_i^*)\cap \text{\rm\textbf{C}}(F_i,\Omega_i)$
with $\text{\rm \textbf{A}}(F_i,\Omega_i,x_i^*)$ given by \eqref{A+} and
\begin{gather}\label{T5.1-3}
\text{\rm\textbf{C}}(F_i,\Omega_i):=\left\{x\in X\bigm| \mathcal{T}^{F_i}_{\Omega_i}(x)=\max_{1\le k\le n}\mathcal{T}^{F_k}_{\Omega_k}(x)\right\}\;\;(i=1,\ldots,n).
\end{gather}
\end{enumerate}
\end{theorem}

\begin{proof}[\normalsize Proof]\normalsize
By Remark~\ref{R2}, $\|\cdot\|^*_{\psi_\infty}=\|\cdot\|_{\psi_1}$.
By Theorem~\ref{T3.5}\eqref{T3.5-1},
$\bx\in\operatorname{Sol}(P)$ if and only if
there exist
$\al_i\ge0$ and $x_i^*\in\partial\mathcal{T}^{F_i}_{\Omega_i}(\bx)$
$(i=1,\ldots,n)$ such that
\begin{gather}\label{T5.1-10}
\sum_{i=1}^n\al_ix_i^*=0_{X^*},\;\;
\sum_{i=1}^n\al_i\le1,\;\;
\sum_{i=1}^n\al_i\mathcal{T}^{F_i}_{\Omega_i}(\bx)
=\max_{1\le k\le n}\mathcal{T}^{F_k}_{\Omega_k}(\bx).
\end{gather}
If $\bx\notin\Omega_i$ for some $i=1,\ldots,n$, then
$\sum_{i=1}^n\al_i=1$.
The last equality in \eqref{T5.1-10} implies
\begin{gather*}
\sum_{i=1}^{n}\al_i\left(\mathcal{T}^{F_i}_{\Omega_i}(\bx)-\max_{1\le k\le n}\mathcal{T}^{F_k}_{\Omega_k}(\bx)\right)=0,
\end{gather*}	
and consequently, condition \eqref{T5.1-6}  holds true.
If $\bx\in\bigcap_{i=1}^n\Omega_i$, then $0_{X^*}\in \partial \mathcal{T}^{F_i}_{\Omega_i}(\bx)$ $(i=1,\ldots,n)$.
In this case, condition \eqref{T5.1-6} holds true with $x^*_i:=0_{X^*}$ and any $\al_i\ge 0$ $(i=1,\ldots,n)$ satisfying 
$\sum_{i=1}^n\al_i=1$.
It is clear that \eqref{T5.1-6} implies \eqref{T5.1-10}.
This proves assertion \eqref{T5.1-1}.

Let $\bx\in\operatorname{Sol}(P)$, $\al_i$ and $x_i^*$ $(i=1,\ldots,n)$ satisfy the conditions in assertion \eqref{T5.1-1}.
By Theorem~\ref{T3.5}\eqref{T3.5-2},
$x\in\operatorname{Sol}(P)$ if and only if
\begin{gather}\label{T5.1-4}
x_i^*\in\partial\mathcal{T}^{F_i}_{\Omega_i}(x)
\;\;\text{for every }i\text{ with }\al_i>0\;\text{and}\;
\sum_{i=1}^n\al_i\mathcal{T}^{F_i}_{\Omega_i}(x)
=\max_{1\le k\le n}\mathcal{T}^{F_k}_{\Omega_k}(x).
\end{gather}
Since $\sum_{i=1}^n\al_i=1$, the last equality is equivalent to
\begin{gather*}
\mathcal{T}^{F_i}_{\Omega_i}(x)
=\max_{1\le k\le n}\mathcal{T}^{F_k}_{\Omega_k}(x)
\;\text{for every }i\text{ with }\al_i>0.
\end{gather*}
This proves \eqref{T5.1-2}.
\end{proof}

\begin{remark}
\begin{enumerate}[\rm (i)]
\item
When $F_1=\cdots=F_n=\B_X$, Theorem~\ref{T5.1} strengthens the corresponding result in \cite{HieCuo26}. 
In addition, if $\Omega_1,\ldots,\Omega_n$ are singletons, it improves \cite[Theorem~6.1]{Cuo26b}.
\item
One can verify that $\bigcup_{i=1}^n\mathbf{C}(F_i,\Omega_i)=X$ and 
$\mathbf{C}(F_i,\Omega_i)$  $(i=1,\ldots,n)$ are closed.
\end{enumerate}
\end{remark}	

The next example computes the sets \eqref{T5.1-3} for  $p$-norms with $p\in[1,\infty]$.
\begin{example}\label{E5.3}
Let $X:=\R^2$ be equipped with the norm $\|\cdot\|:=\|\cdot\|_{\psi_p}$ for $p\in[1,\infty]$,  and   $F_i,\Omega_i$ $(i=1,2)$ be given by \eqref{E4.4-2}.
By Example~\ref{E3.7}\eqref{E4.4.2}, $\mathcal{T}^{F_i}_{\Omega_i}(x)=[\phi_i(x-x_i)-1]^+$ for all $x\in X$ $(i=1,2)$.
Then
\begin{align*}
\text{\rm\textbf{C}}(F_1,\Omega_1)
&=\left\{u\in\R^2\mid\mathcal{T}^{F_1}_{\Omega_1}(u)\ge \mathcal{T}^{F_2}_{\Omega_2}(u)\right\}\\
&=\left\{u\in\R^2\mid[\phi_1(u-x_1)-1]^+\ge
[\phi_2(u-x_2)-1]^+\right\}\\
&=\left\{u\in\R^2\mid\phi_1(u-x_1)\ge\phi_2(u-x_2)\right\}\cup\left\{u\in\R^2\mid\phi_2(u-x_2)\le 1\right\}\\
&=\left\{u\in\R^2\mid\phi_1(u-x_1)\ge\phi_2(u-x_2)\right\}\\
&=\left\{(u_1,u_2)\mid\|(u_1+1,u_2)\|\ge\max\{\|(u_1-3,u_2)\|,2|u_2|\}\right\}.
\end{align*}	
Similarly,
\begin{align*}
\text{\rm\textbf{C}}(F_2,\Omega_2)
&=\left\{u\in\R^2\mid\phi_2(u-x_2)\ge\phi_1(u-x_1)\right\}\\
&=\left\{(u_1,u_2)\mid\max\{\|(u_1-3,u_2)\|,2|u_2|\}\ge\|(u_1+1,u_2)\|\right\}.
\end{align*}
\begin{enumerate}[\rm (i)]
\item\label{E5.3-1}	
If $p:=\infty$, then
\begin{align*}
\text{\rm\textbf{C}}(F_1,\Omega_1)
&=\left\{(u_1,u_2)\mid\max\{|u_1+1|,|u_2|\}\ge\max\{|u_1-3|,2|u_2|\}\right\}\\
&=\left\{(u_1,u_2)\mid u_1\ge 1,\;|u_2|\le\frac{u_1+1}{2}\right\},  \\
\text{\rm\textbf{C}}(F_2,\Omega_2)
&=\left\{(u_1,u_2)\mid\max\{|u_1-3|,2|u_2|\}\ge\max\{|u_1+1|,|u_2|\}\right\}\\
&=\{(u_1,u_2)\mid u_1\le 1\}\cup
\left\{(u_1,u_2)\mid |u_2|\ge\frac{u_1+1}{2}\right\}. 
\end{align*}	
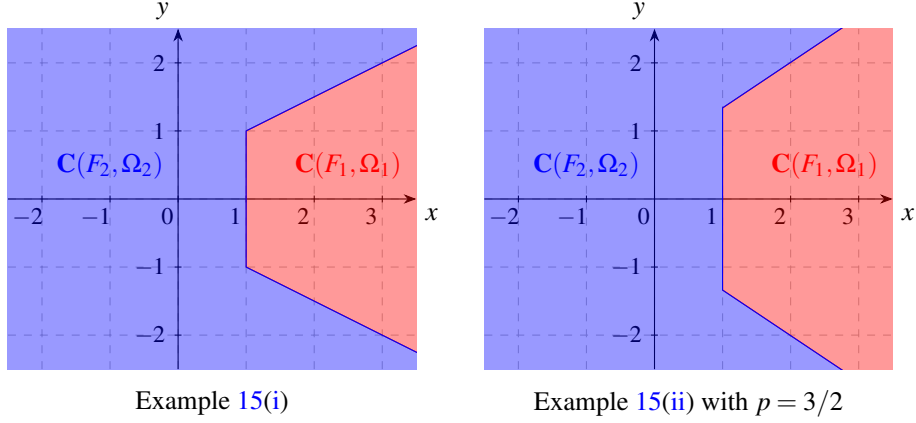
\begin{figure}[H]
\centering
\begin{tikzpicture}[scale=0.9]
\begin{scope}
\draw[gray!50,dashed] (-2.5,-2.5) grid (3.5,2.5);
			
\foreach \x in {0,1,2,3}{
\draw (\x,0.05)--(\x,-0.05) node[anchor=north,xshift=-4pt] at (\x,0) {\small $\x$};}

\foreach \x in {-2,-1}
{\draw (\x,0.05)--(\x,-0.05) node[anchor=north east,xshift=4pt] at (\x,0) {\small $\x$};}

\foreach \y in {-2,-1,1,2}
\draw (0.05,\y)--(-0.05,\y) node[left] {\small $\y$};

\draw[-{Stealth}] (-2.5,0)--(3.5,0) node[below right] {$x$};
\draw[-{Stealth}] (0,-2.5)--(0,2.5) node[above left] {$y$};

\node at (0.5,-3) {Example~\ref{E5.3}\eqref{E5.3-1}};	
			
\clip (-2.5,-2.5) rectangle (3.5,2.5);
			
\fill[red, fill opacity=0.4]
(5,3)--(1,1)--(1,-1)--(5,-3);
\draw[red] (5,3)--(1,1)--(1,-1)--(5,-3);
\node[red] at (2.5,0.5) {$\text{\rm\textbf{C}}(F_1,\Omega_1)$};	
			
\fill[blue, fill opacity=0.4]
(4,-3)--(-3,-3)--(-3,3)--(4,3)--(4,2.5)--(1,1)--(1,-1)--(4,-2.5)--(4,-3);
\draw[blue] (4,-3)--(-3,-3)--(-3,3)--(4,3)--(4,2.5)--(1,1)--(1,-1)--(4,-2.5)--(4,-3);
\node[blue] at (-1,0.5) {$\text{\rm\textbf{C}}(F_2,\Omega_2)$};		
\end{scope}

\begin{scope}[shift={(7,0)}]
\draw[gray!50,dashed] (-2.5,-2.5) grid (3.5,2.5);
	
\foreach \x in {0,1,2,3}{
\draw (\x,0.05)--(\x,-0.05) node[anchor=north,xshift=-4pt] at (\x,0) {\small $\x$};}
	
\foreach \x in {-2,-1}
{\draw (\x,0.05)--(\x,-0.05) node[anchor=north east,xshift=4pt] at (\x,0) {\small $\x$};}
	
\foreach \y in {-2,-1,1,2}
\draw (0.05,\y)--(-0.05,\y) node[left] {\small $\y$};
	
\draw[-{Stealth}] (-2.5,0)--(3.5,0) node[below right] {$x$};
\draw[-{Stealth}] (0,-2.5)--(0,2.5) node[above left] {$y$};

\node at (0.5,-3) {Example~\ref{E5.3}\eqref{E5.3-2} with $p=3/2$};	
	
\clip (-2.5,-2.5) rectangle (3.5,2.5);
	
\fill[red, fill opacity=0.4]
(5,{6/pow(pow(2,1.5)-1, 2/3)})--(1,{2/pow(pow(2,1.5)-1, 2/3)})--(1,-{2/pow(pow(2,1.5)-1, 2/3)})--(5,-{6/pow(pow(2,1.5)-1, 2/3)});
\draw[red] 
(5,{6/pow(pow(2,1.5)-1, 2/3)})--(1,{2/pow(pow(2,1.5)-1, 2/3)})--(1,-{2/pow(pow(2,1.5)-1, 2/3)})--(5,-{6/pow(pow(2,1.5)-1, 2/3)});
\node[red] at (2.5,0.5) {$\text{\rm\textbf{C}}(F_1,\Omega_1)$};	
	
\fill[blue, fill opacity=0.4]
(4,-3)--(-3,-3)--(-3,3)--(4,3)--(4,{5/pow(pow(2,1.5)-1, 2/3)})--(1,{2/pow(pow(2,1.5)-1, 2/3)})--(1,-{2/pow(pow(2,1.5)-1, 2/3)})--(4,-{5/pow(pow(2,1.5)-1, 2/3)})--(4,-3);
\draw[blue] (4,-3)--(-3,-3)--(-3,3)--(4,3)--(4,{5/pow(pow(2,1.5)-1, 2/3)})--(1,{2/pow(pow(2,1.5)-1, 2/3)})--(1,-{2/pow(pow(2,1.5)-1, 2/3)})--(4,-{5/pow(pow(2,1.5)-1, 2/3)})--(4,-3);
\node[blue] at (-1,0.5) {$\text{\rm\textbf{C}}(F_2,\Omega_2)$};		
\end{scope}
\end{tikzpicture}
\caption{Example~\ref{E5.3}}
\label{fig5}
\end{figure}
\item\label{E5.3-2}	
If $p\in[1,\infty)$, then 
\begin{align*}
\text{\rm\textbf{C}}(F_1,\Omega_1)
&=\left\{(u_1,u_2)\mid u_1\ge 1,\;|u_2|\le\frac{u_1+1}{(2^p-1)^{\frac{1}{p}}}
\right\},\\
\text{\rm\textbf{C}}(F_2,\Omega_2)
&=\{(u_1,u_2)\mid u_1\le 1\}\cup \left\{(u_1,u_2)\mid |u_2|\ge\frac{u_1+1}{(2^p-1)^{\frac{1}{p}}}
\right\}.
\end{align*}
\end{enumerate}
\end{example}	

The following examples illustrate Theorem~\ref{T5.1}.
\begin{example}\label{E5.4}
Let $X:=\R^2$ be equipped with the norm $\|\cdot\|:=\|\cdot\|_{\psi_p}$ with $p\in[1,\infty]$, $\mathbf{u},\mathbf{u}^*,x_1,x_2,\bx$ and $F_i,\Omega_i,\phi_i$ $(i=1,2)$ be given by \eqref{E4.4-1}--\eqref{E4.4-0}.
We have
\begin{gather}\label{E5.4-3}
F(x)=\max\left\{\mathcal{T}^{F_1}_{\Omega_1}(x),\mathcal{T}^{F_2}_{\Omega_2}(x)\right\}
\end{gather}
for all $x\in\R^2$.
By Example~\ref{E3.7}\eqref{E4.4.5}, $\bx\in\operatorname{Sol}(P)$.
Let $x^*_1:=\mathbf{u}$, $x^*_2:=-\mathbf{u}$, and $\al_1=\al_2=\frac{1}{2}$.
By Example~\ref{E4.4}, $x_i^*\in\partial\mathcal{T}^{F_i}_{\Omega_i}(\bx)$ $(i=1,2)$.
By Example~\ref{E3.7}\eqref{E4.4.2}, $\mathcal{T}^{F_i}_{\Omega_i}(\bx)=\phi_i(\bx-x_i)-1=1$ $(i=1,2)$. 
Thus, condition \eqref{T5.1-6} is satisfied.
\begin{enumerate}[\rm (i)]
\item\label{E5.4-1}
Let $p:=\infty$.
By \eqref{E4.4-8}, \eqref{E4.4-9}, and Example~\ref{E5.3}\eqref{E5.3-1},
\begin{align*}
\text{\rm \textbf{B}}(F_1,\Omega_1,x_1^*)
&=\text{\rm \textbf{A}}(F_1,\Omega_1,x_1^*)\cap \text{\rm\textbf{C}}(F_1,\Omega_1)
=\left\{(u_1,u_2)\mid u_1\ge 1,\;|u_2|\le\frac{1+u_1}{2}\right\},\\
\text{\rm \textbf{B}}(F_2,\Omega_2,x_2^*)
&=\text{\rm \textbf{A}}(F_2,\Omega_2,x_2^*)\cap \text{\rm\textbf{C}}(F_2,\Omega_2)
=\left\{(u_1,u_2)\mid u_1\le 1,\;|u_2|\le \frac{3-u_1}{2}\right\}.
\end{align*}	
By \eqref{SCC},
\begin{gather*}
 \text{\rm Sol}(P)=\text{\rm \textbf{B}}(F_1,\Omega_1,x_1^*)\cap \text{\rm \textbf{B}}(F_2,\Omega_2,x_2^*)=
\{(1,u)\mid |u|\le 1\}.
\end{gather*}
\begin{figure}[H]
\centering
\begin{tikzpicture}[scale=0.9]
\begin{scope}
\draw[gray!50,dashed] (-1.99,-2.5) grid (3.99,2.5);
			
\foreach \x in {0,1,2,3}{
\draw (\x,0.05)--(\x,-0.05) node[anchor=north,xshift=-4pt] at (\x,0) {\small $\x$};}
			
\foreach \x in {-1}
{\draw (\x,0.05)--(\x,-0.05) node[anchor=north east,xshift=4pt] at (\x,0) {\small $\x$};}
			
\foreach \y in {-2,-1,1,2}
\draw (0.05,\y)--(-0.05,\y) node[left] {\small $\y$};
			
\draw[-{Stealth}] (-2,0)--(4,0) node[below right] {$x$};
\draw[-{Stealth}] (0,-2.5)--(0,2.5) node[above left] {$y$};

\clip (-2,-2.5) rectangle (4,2.5);
			
\fill[red, thick, fill opacity=0.4]
(5,3)--(1,1)--(1,-1)--(5,-3);
\draw[red, thick] (5,3)--(1,1)--(1,-1)--(5,-3);
\node[red, font=\small] at (3,0.5) {$\text{\rm\textbf{B}}(F_1,\Omega_1, x_1^*)$};	
			
\fill[blue, thick, fill opacity=0.4]
(-5,4)--(1,1)--(1,-1)--(-5,-4);
\draw[blue, thick] (-5,4)--(1,1)--(1,-1)--(-5,-4);
\node[blue, font=\small] at (-1,0.5) {$\text{\rm\textbf{B}}(F_2,\Omega_2,x_2^*)$};		

\draw[violet, thick] (1,1)--(1,-1);
\fill(1,1) circle (1pt);
\fill(1,-1) circle (1pt);
\end{scope}
		
\begin{scope}[shift={(7,0)}]
\draw[gray!50,dashed] (-1.99,-2.5) grid (3.99,2.5);

\foreach \x in {0,1,2,3}{
\draw (\x,0.05)--(\x,-0.05) node[anchor=north,xshift=-4pt] at (\x,0) {\small $\x$};}

\foreach \x in {-1}
{\draw (\x,0.05)--(\x,-0.05) node[anchor=north east,xshift=4pt] at (\x,0) {\small $\x$};}

\foreach \y in {-2,-1,1,2}
\draw (0.05,\y)--(-0.05,\y) node[left] {\small $\y$};

\draw[-{Stealth}] (-2,0)--(4,0) node[below right] {$x$};
\draw[-{Stealth}] (0,-2.5)--(0,2.5) node[above left] {$y$};

\draw[-{Stealth}, red, very thick] (1,0)--(4,0);
\node[red, font=\small] at (3,0.4)
{$\text{\rm\textbf{B}}(F_1,\Omega_1,x_1^*)$};

\draw[blue, very thick] (-2,0)--(1,0);
\node[blue, font=\small] at (-1,0.4) {$\text{\rm \textbf{B}}(F_2,\Omega_2,x_2^*)$};

\fill(1,0) circle (2pt);
\node at (1.2,0.3) {$\bar{x}$};
\end{scope}
		
\end{tikzpicture}
\caption{Example~\ref{E5.4}\eqref{E5.4-1} and \eqref{E5.4-2}}
\label{fig6}
\end{figure}
\item\label{E5.4-2}
Let $p\in[1,\infty)$.
By \eqref{E4.4-11} and Example~\ref{E5.3}\eqref{E5.3-2},
\begin{align*}
\text{\rm \textbf{B}}(F_1,\Omega_1,x_1^*)
&=\text{\rm \textbf{A}}(F_1,\Omega_1,x_1^*)\cap \text{\rm\textbf{C}}(F_1,\Omega_1)=\left\{(u,0)\mid u\ge 1\right\},\\
\text{\rm \textbf{B}}(F_2,\Omega_2,x_2^*)&=\text{\rm \textbf{A}}(F_2,\Omega_2,x_2^*)\cap \text{\rm\textbf{C}}(F_2,\Omega_2)=\left\{(u,0)\mid u\le 1\right\}.
\end{align*}	
By \eqref{SCC}, 
$\text{\rm Sol}(P)=\text{\rm \textbf{B}}(F_1,\Omega_1,x_1^*)\cap \text{\rm \textbf{B}}(F_2,\Omega_2,x_2^*)=
\{\bx\}.$
\end{enumerate}
\end{example}	

\begin{example}
Let $X:=L^2[0,1]$ be the Hilbert space in Example~\ref{E4.9}.
Let $\mathbf{u},\mathbf{u}^*$,  $x_1,x_2,\bx$ and $F_i,\Omega_i,\phi_i$ $(i=1,2)$ be given by \eqref{E4.9-1}--\eqref{E4.9-4}.
The objective function is given by \eqref{E5.4-3} for all $x\in X$.
By Example~\ref{E3.7}\eqref{E4.4.5}, $\bar x\in\text{\rm Sol}(P)$. 
By Example~\ref{E3.7}\eqref{E4.4.2},
\begin{gather}\label{E4.13-1}
\mathcal{T}^{F_1}_{\Omega_1}(x)=\left[\|x+\mathbf{u}/\sqrt 3\|-\frac{1}{2\sqrt 3}\right]^+,\\ \label{E4.13-2}
\mathcal{T}^{F_2}_{\Omega_2}(x)=\left[\max\{\|x-\mathbf{u}/\sqrt 3\|,2|\langle\mathbf{u}^*,x\rangle|\}-\frac{1}{2\sqrt 3}\right]^+
\end{gather}	
for all $x\in X$.
In particular,
$\mathcal{T}^{F_1}_{\Omega_1}(\bx)=\mathcal{T}^{F_2}_{\Omega_2}(\bx)=\frac{1}{2\sqrt 3}$.
Let $x^*_1:=\mathbf{u}$, $x^*_2:=-\mathbf{u}$, and 
$\al_1=\al_2=\frac{1}{2}$.
By Example~\ref{E4.9}, $x_i^*\in\partial\mathcal{T}^{F_i}_{\Omega_i}(\bx)$ $(i=1,2)$. 
It is clear that condition \eqref{T5.1-6} is satisfied.
By \eqref{T5.1-3}, \eqref{E4.13-1} and \eqref{E4.13-2},
\begin{gather*}
\text{\rm\textbf{C}}(F_1,\Omega_1)=\{x\in X\mid\langle x,\mathbf{u}\rangle\ge 0,\;2|\langle\mathbf{u}^*,x\rangle|\le\|x+\mathbf{u}/\sqrt 3\| \},\\
\text{\rm\textbf{C}}(F_2,\Omega_2)=\{x\in X\mid \langle x,\mathbf{u}\rangle\le 0\}\cup\{x\in X\mid 2|\langle\mathbf{u}^*,x\rangle|\ge\|x+\mathbf{u}/\sqrt 3\| \}.
\end{gather*}	
The sets $\mathbf A(F_1,\Omega_1,x_1^*)$ and $\mathbf A(F_2,\Omega_2,x_2^*)$ are given by \eqref{E4.9-5}. 
Then
\begin{align*}
\text{\rm \textbf{B}}(F_1,\Omega_1,x_1^*)
&=\text{\rm \textbf{A}}(F_1,\Omega_1,x_1^*)\cap \text{\rm\textbf{C}}(F_1,\Omega_1)=\{\lambda\mathbf{u}\mid \lambda\ge 0\},\\
\text{\rm \textbf{B}}(F_2,\Omega_2,x_2^*)
&=\text{\rm \textbf{A}}(F_2,\Omega_2,x_2^*)\cap \text{\rm\textbf{C}}(F_2,\Omega_2)=\{\lambda\mathbf{u}\mid \lambda\le 0\}.
\end{align*}	
By \eqref{SCC}, 
$\text{\rm Sol}(P)=\text{\rm \textbf{B}}(F_1,\Omega_1,x_1^*)\cap \text{\rm \textbf{B}}(F_2,\Omega_2,x_2^*)=
\{\bx\}$.
\end{example}	

\section{$p$-Fermat--Torricelli Problem}\label{S7}
We study  the case
$\Phi(u):=\|u\|_{\psi_p}$ for all $u\in\R^n_+$ and some
$p\in(1,\infty)$.
Then
\begin{gather*}
F(x)=\left[\left(\mathcal{T}^{F_1}_{\Omega_1}(x)\right)^p+\cdots+\left(\mathcal{T}^{F_n}_{\Omega_n}(x)\right)^p\right]^{\frac{1}{p}}\;\;\text{for all}\;\;x\in X.
\end{gather*}	

\begin{theorem}\label{T3.16}
Let $\bx\in X$.
The following assertions hold.
\begin{enumerate}[\rm (i)]
\item\label{T3.16.1}
$\bx\in\operatorname{Sol}(P)$ if and only if there exist
$x_i^*\in\partial\mathcal{T}^{F_i}_{\Omega_i}(\bx)$ $(i=1,\ldots,n)$
such that
\begin{gather}\label{T3.16-4}
\sum_{i=1}^n
\left(\mathcal{T}^{F_i}_{\Omega_i}(\bx)\right)^{p-1}x_i^*
=0_{X^*}.
\end{gather}
\item\label{T3.16.2}
Let $\bx\in\operatorname{Sol}(P)$, and let
$x_1^*,\ldots,x_n^*$ satisfy the conditions in
assertion \eqref{T3.16.1}.
Then
\begin{gather}\label{T3.16-5}
\operatorname{Sol}(P)
=\bigcap_{i=1}^n\text{\rm\textbf{D}}(F_i,\Omega_i,x_i^*),
\end{gather}
where
\begin{gather*}
\text{\rm\textbf{D}}(F_i,\Omega_i,x_i^*)
:=
\begin{cases}
\text{\rm\textbf{A}}(F_i,\Omega_i,x_i^*)
\cap\text{\rm\textbf{E}}(F_i,\Omega_i,\bx)
&\text{if}\;\;\mathcal{T}^{F_i}_{\Omega_i}(\bx)>0,\\
\text{\rm\textbf{E}}(F_i,\Omega_i,\bx)
&\text{if}\;\;\mathcal{T}^{F_i}_{\Omega_i}(\bx)=0
\end{cases}
\end{gather*}
with $\text{\rm\textbf{A}}(F_i,\Omega_i,x_i^*)$
given by \eqref{A+} and
\begin{gather}\label{T6.1-1}
\text{\rm\textbf{E}}(F_i,\Omega_i,\bx)
:=\left\{x\in X\mid
\mathcal{T}^{F_i}_{\Omega_i}(x)=\mathcal{T}^{F_i}_{\Omega_i}(\bx)\right\}
\;\;(i=1,\ldots,n).
\end{gather}
\end{enumerate}
\end{theorem}

\begin{proof}[\normalsize Proof]\normalsize
By Remark~\ref{R2}, $\|\cdot\|^*_{\psi_p}=\|\cdot\|_{\psi_q}$
with $q\in(1,\infty)$ satisfying $\frac1p+\frac1q=1$.
By Theorem~\ref{T3.5}\eqref{T3.5-1},
$\bx\in\operatorname{Sol}(P)$ if and only if there exist
$\al_i\ge0$ and $x_i^*\in\partial\mathcal{T}^{F_i}_{\Omega_i}(\bx)$
$(i=1,\ldots,n)$ such that
\begin{gather}\label{T6.1-10}
\sum_{i=1}^n\al_ix_i^*=0_{X^*},\;\;
\sum_{i=1}^n\al_i^q\le1,\;\;
\sum_{i=1}^n\al_i\mathcal{T}^{F_i}_{\Omega_i}(\bx)=F(\bx).
\end{gather}
If $F(\bx)>0$, then $\sum_{i=1}^n\al_i^q=1$.
The last equality in \eqref{T6.1-10} implies
\begin{gather}\label{T6.1-3}
\al_i=
\left(\frac{\mathcal{T}^{F_i}_{\Omega_i}(\bx)}{F(\bx)}\right)^{p-1}
\ge0
\;\;(i=1,\ldots,n).
\end{gather}
Substituting these multipliers into the first equality in
\eqref{T6.1-10} gives \eqref{T3.16-4}.
Conversely, \eqref{T3.16-4} implies \eqref{T6.1-10}
with the multipliers given by \eqref{T6.1-3}.
If $F(\bx)=0$, then $\bx\in\bigcap_{i=1}^n\Omega_i$
is a solution.
Thus, $0_{X^*}\in\partial\mathcal{T}^{F_i}_{\Omega_i}(\bx)$
$(i=1,\ldots,n)$, and consequently, \eqref{T3.16-4} holds with
$x_i^*:=0_{X^*}$ $(i=1,\ldots,n)$.
This proves \eqref{T3.16.1}.
	
Let $\bx\in\operatorname{Sol}(P)$ and 
$x_1^*,\ldots,x_n^*$ satisfy the conditions in
assertion \eqref{T3.16.1}.
Suppose that $F(\bx)>0$.
Let $\al_i$ $(i=1,\ldots,n)$ be given by \eqref{T6.1-3}.
Then $\al_i>0$ if and only if
$\mathcal{T}^{F_i}_{\Omega_i}(\bx)>0$ $(i=1,\ldots,n)$.
Let $x\in\operatorname{Sol}(P)$.
By Theorem~\ref{T3.5}\eqref{T3.5-2},
\begin{gather}\label{T5.1-7}
x_i^*\in\partial\mathcal{T}^{F_i}_{\Omega_i}(x)
\;\;\text{for every }i\text{ with }\al_i>0\;\;\text{and}\;\;
\sum_{i=1}^n\al_i\mathcal{T}^{F_i}_{\Omega_i}(x)=F(x).
\end{gather}
Since $F(x)=F(\bx)>0$, the equality in \eqref{T5.1-7} gives
\begin{gather}\label{T6.1.3}
\al_i=
\left(\frac{\mathcal{T}^{F_i}_{\Omega_i}(x)}{F(x)}\right)^{p-1}
\;\;(i=1,\ldots,n).
\end{gather}
By \eqref{T6.1-3} and \eqref{T6.1.3},
$\mathcal{T}^{F_i}_{\Omega_i}(x)=\mathcal{T}^{F_i}_{\Omega_i}(\bx)$
$(i=1,\ldots,n)$.
From this and the first condition in \eqref{T5.1-7},
$x\in \text{\rm\textbf{A}}(F_i,\Omega_i,x_i^*)
\cap\text{\rm\textbf{E}}(F_i,\Omega_i,\bx)$
if $\mathcal{T}^{F_i}_{\Omega_i}(\bx)>0$, and 
$x\in \text{\rm\textbf{E}}(F_i,\Omega_i,\bx)=\Omega_i$
if $\mathcal{T}^{F_i}_{\Omega_i}(\bx)=0$.
Thus,
$x\in\bigcap_{i=1}^n\text{\rm\textbf{D}}(F_i,\Omega_i,x_i^*)$.
Conversely, if \(x\) belongs to this intersection, then
$\mathcal{T}^{F_i}_{\Omega_i}(x)=\mathcal{T}^{F_i}_{\Omega_i}(\bx)$
$(i=1,\ldots,n)$, and hence $F(x)=F(\bx)$.
Thus, $x\in\operatorname{Sol}(P)$.
If $F(\bx)=0$, then $\mathcal{T}^{F_i}_{\Omega_i}(\bx)=0$ $(i=1,\ldots,n)$, and consequently,
\begin{gather*}
\text{\rm\textbf{D}}(F_i,\Omega_i,x_i^*)
=\text{\rm\textbf{E}}(F_i,\Omega_i,\bx)=\Omega_i
\;\;(i=1,\ldots,n).
\end{gather*}
Since $\operatorname{Sol}(P)=\bigcap_{i=1}^n\Omega_i$,
formula \eqref{T3.16-5} also holds in this case.
\end{proof}

\begin{remark}
When $F_1=\cdots=F_n=\B_X$, Theorem~\ref{T3.16} improves the corresponding result in \cite{HieCuo26}.
In addition, if $\Omega_1,\ldots
,\Omega_n$ are singletons, it strengthens \cite[Theorem~7.1]{Cuo26b}.
\end{remark}	

The following examples illustrate Theorem~\ref{T3.16}.
\begin{example}\label{E7.2}
Let $X:=\R^2$ be equipped with the norm $\|\cdot\|:=\|\cdot\|_{\psi_{p'}}$ with $p'\in[1,\infty]$, $\mathbf{u},\mathbf{u}^*,x_1,x_2,\bx$ and $F_i,\Omega_i,\phi_i$ $(i=1,2)$ be given by \eqref{E4.4-1}--\eqref{E4.4-0}.
By Example~\ref{E3.7}\eqref{E4.4.2}, $\mathcal{T}^{F_i}_{\Omega_i}(\cdot)=[\phi_i(\cdot-x_i)-1]^+$ $(i=1,2)$.
We have
\begin{gather}\label{E7.2-3}
F(x)=\left[\left(\mathcal{T}^{F_1}_{\Omega_1}(x)\right)^p+\left(\mathcal{T}^{F_2}_{\Omega_2}(x)\right)^p\right]^{\frac{1}{p}}
\end{gather}	
for all $x\in\R^2$.
By Example~\ref{E3.7}\eqref{E4.4.5}, we have $\bar x\in\text{\rm Sol}(P)$.
Let $x^*_1:=\mathbf{u}$ and $x^*_2:=-\mathbf{u}$.
By Example~\ref{E5.4}, $x_i^*\in\partial\mathcal{T}^{F_i}_{\Omega_i}(\bx)$ and $\mathcal{T}^{F_i}_{\Omega_i}(\bx)=1$ $(i=1,2)$.
It is clear that condition \eqref{T3.16-4} is satisfied.
\begin{enumerate}[\rm (i)]
\item\label{E7.2-1}
Let $p':=\infty$.
By \eqref{T6.1-1},
\begin{align*}
\text{\rm \textbf{E}}(F_1,\Omega_1,\bar x)&=\left\{u\in\R^2\mid \mathcal{T}^{F_1}_{\Omega_1}(u)=1\right\}
=\left\{(u_1,u_2)\mid\max\{|u_1+1|,|u_2|\}=2\right\}\\
&=\left\{(u_1,u_2)\mid |u_1+1|=2, |u_2|\le 2\right\}\cup\left\{(u_1,u_2)\mid |u_1+1|\le2, |u_2|=2\right\},\\
\text{\rm \textbf{E}}(F_2,\Omega_2,\bar x)&=\left\{u\in\R^2\mid \mathcal{T}^{F_2}_{\Omega_2}(u)=1\right\}
=\left\{(u_1,u_2)\mid\max\{|u_1-3|,2|u_2|\}=2\right\}\\
&=\left\{(u_1,u_2)\mid |u_1-3|=2, |u_2|\le 1\right\}\cup\left\{(u_1,u_2)\mid |u_1-3|\le2, |u_2|=1\right\}.
\end{align*}	
The sets $\text{\rm \textbf{A}}(F_1,\Omega_1,x_1^*)$ and $\text{\rm \textbf{A}}(F_2,\Omega_2,x_2^*)$
are given by \eqref{E4.4-8} and \eqref{E4.4-9}.
Then
\begin{align*}
\text{\rm \textbf{D}}(F_1,\Omega_1,x_1^*)
&=\text{\rm \textbf{A}}(F_1,\Omega_1,x_1^*)\cap \text{\rm\textbf{E}}(F_1,\Omega_1,\bar x)
=\{(1,u)\mid |u|\le 2\},\\
\text{\rm \textbf{D}}(F_2,\Omega_2,x_2^*)
&=\text{\rm \textbf{A}}(F_2,\Omega_2,x_2^*)\cap \text{\rm\textbf{E}}(F_2,\Omega_2,\bar x)
=\{(1,u)\mid |u|\le 1\}.
\end{align*}	
By \eqref{T3.16-5}, 
$\text{\rm Sol}(P)=\text{\rm \textbf{D}}(F_1,\Omega_1,x_1^*)\cap \text{\rm \textbf{D}}(F_2,\Omega_2,x_2^*)=\{(1,u)\mid |u|\le 1\}$.
\begin{figure}[H]
\centering
\begin{tikzpicture}[scale=0.8]
\begin{scope}
\draw[step=1cm,gray!50,dashed] (-3.5,-2.5) grid (3.5,2.5);

\foreach \x in {1,2}{
\draw (\x,0.05)--(\x,-0.05) node[anchor=north,xshift=4pt] at (\x,0) {\small $\x$};}

\foreach \x in {-2,-1}
{\draw (\x,0.05)--(\x,-0.05) node[anchor=north east,xshift=4pt] at (\x,0) {\small $\x$};}

\foreach \x in {0}
{\draw (\x,0.05)--(\x,-0.05) node[anchor=north east,xshift=3pt] at (\x,0) {\small $\x$};}

\foreach \y in {-1,1}
\draw (0.05,\y)--(-0.05,\y) node[left] {\small $\y$};
\draw[-{Stealth}](-3.5,0)--(3.5,0) node[below right] {$x$};
\draw[-{Stealth}](0,-2.5)--(0,2.5) node[above left] {$y$};

\clip (-3.5,-2.5) rectangle (3.5,2.5);
	
\fill[blue, fill opacity=0.4]
(4,3)--(2,3)--(0,1)--(0,-1)--(2,-3)--(4,-3);
\draw[blue, thick] (4,3)--(2,3)--(0,1)--(0,-1)--(2,-3)--(4,-3);
\node[blue, font=\small] at (2.4,1.5) {$\text{\rm \textbf{A}}(F_1,\Omega_1,x_1^*)$};

\draw[red, thick] (-3,-2) rectangle (1,2);
\node[red, font=\small] at (-1.5,1.5) {$\text{\rm \textbf{E}}(F_1,\Omega_1,\bar{x})$};
	
\draw[violet, very thick] (1,2)--(1,-2);
\fill(1,2) circle (1pt);
\fill(1,-2) circle (1pt);
\node[violet, font=\small] at (2.2,-1.5) {$\text{\rm\textbf{D}}(F_1,\Omega_1,x_1^*)$};
\end{scope}

\begin{scope}[shift={(6,0)}]
\draw[step=1cm,gray!50,dashed] (-1.5,-2.5) grid (5.5,2.5);

\foreach \x in {1,2,3,4}{
\draw (\x,0.05)--(\x,-0.05) node[anchor=north,xshift=4pt] at (\x,0) {\small $\x$};}

\foreach \x in {-1}
{\draw (\x,0.05)--(\x,-0.05) node[anchor=north east,xshift=4pt] at (\x,0) {\small $\x$};}

\foreach \x in {0}
{\draw (\x,0.05)--(\x,-0.05) node[anchor=north east,xshift=3pt] at (\x,0) {\small $\x$};}

\foreach \y in {-2,-1,1,2}
\draw (0.05,\y)--(-0.05,\y) node[left] {\small $\y$};
\draw[-{Stealth}](-1.5,0)--(5.5,0) node[below right] {$x$};
\draw[-{Stealth}](0,-2.5)--(0,2.5) node[above left] {$y$};

\clip (-1.5,-2.5) rectangle (5.5,2.5);

\fill[blue, fill opacity=0.4]
(-3,3)--(2,0.5)--(2,-0.5)--(-3,-3);
\draw[blue, thick] (-3,3)--(2,0.5)--(2,-0.5)--(-3,-3);
\node[blue, font=\small] at (-0.4,0.5) {$\text{\rm \textbf{A}}(F_2,\Omega_2,x_2^*)$};

\draw[red, thick] (1,-1) rectangle (5,1);
\node[red, font=\small] at (3,1.5) {$\text{\rm \textbf{E}}(F_2,\Omega_2,\bar{x})$};

\draw[violet, very thick] (1,1)--(1,-1);
\fill(1,1) circle (1pt);
\fill(1,-1) circle (1pt);
\node[violet, font=\small] at (2.3,-1.5) {$\text{\rm\textbf{D}}(F_2,\Omega_2,x_2^*)=\text{\rm Sol}(P)$};
\end{scope}

\end{tikzpicture}
\caption{Example~\ref{E7.2}\eqref{E7.2-1}}
\label{fig7}
\end{figure}
\item\label{E7.2-2}
Let $p'\in[1,\infty)$.
Then
\begin{align*}
\text{\rm \textbf{E}}(F_1,\Omega_1,\bar x)
&=\left\{(u_1,u_2)\mid |u_1+1|^{p'}+|u_2|^{p'}=2^{p'}\right\},\\
\text{\rm \textbf{E}}(F_2,\Omega_2,\bar x)
&=\left\{(u_1,u_2)\mid |u_1-3|^{p'}+|u_2|^{p'}=2^{p'},\; |u_2|\le 1\right\}\\
&\hspace{0.5cm}\cup\left\{(u_1,u_2)\mid |u_1-3|\le(2^{p'}-1)^{1/p'},\; |u_2|=1\right\}.
\end{align*}	
The sets $\text{\rm \textbf{A}}(F_1,\Omega_1,x_1^*)$ and $\text{\rm \textbf{A}}(F_2,\Omega_2,x_2^*)$
are given by \eqref{E4.4-11}.
Then 
\begin{align*}
\text{\rm \textbf{D}}(F_1,\Omega_1,x_1^*)
&=\text{\rm \textbf{A}}(F_1,\Omega_1,x_1^*)\cap \text{\rm\textbf{E}}(F_1,\Omega_1,\bar x)=\{\bx\},\\
\text{\rm \textbf{D}}(F_2,\Omega_2,x_2^*)&=\text{\rm \textbf{A}}(F_2,\Omega_2,x_2^*)\cap \text{\rm\textbf{E}}(F_2,\Omega_2,\bar x)
=\{\bx\}.
\end{align*}	
By \eqref{T3.16-5}, 
$\text{\rm Sol}(P)=\text{\rm \textbf{D}}(F_1,\Omega_1,x_1^*)\cap \text{\rm \textbf{D}}(F_2,\Omega_2,x_2^*)=\{\bx\}$.
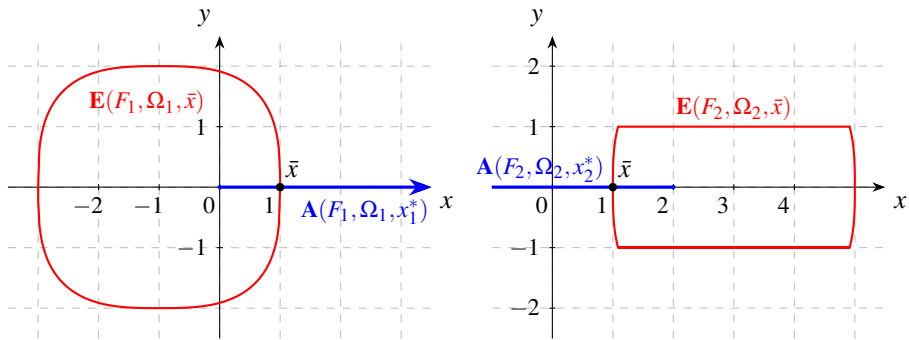
\begin{figure}[H]
\centering
\begin{tikzpicture}[scale=0.8]
\begin{scope}
\draw[gray!50,dashed] (-3.5,-2.5) grid (3.5,2.5);
			
\foreach \x in {-2,-1,0,1}{
\draw (\x,0.05)--(\x,-0.05) node[anchor=north,xshift=-4pt] at (\x,0) {\small $\x$};}

\foreach \y in {-1,1}
\draw (0.05,\y)--(-0.05,\y) node[left] {\small $\y$};
			
\draw[-{Stealth}] (-3.5,0)--(3.5,0) node[below right] {$x$};
\draw[-{Stealth}] (0,-2.5)--(0,2.5) node[above left] {$y$};

\clip (-3.5,-2.5) rectangle (3.5,2.5);
			
\draw[red, thick]
plot[domain=-3:1,samples=200]
(\x,{(8-abs(\x+1)^3)^(1/3)})--
plot[domain=1:-3,samples=200]
(\x,{-(8-abs(\x+1)^3)^(1/3)})--cycle;
\node[red, font=\small] at (-1.2,1.4) {$\text{\rm \textbf{E}}(F_1,\Omega_1,\bar x)$};	

\draw[-{Stealth}, blue, very thick] (0,0)--(3.5,0);
\node[blue, font=\small] at (2.4,-0.4)
{$\text{\rm\textbf{A}}(F_1,\Omega_1,x_1^*)$};
\fill[blue](0,0) circle (1pt);

\fill(1,0) circle (2pt);
\node at (1.2,0.3) {$\bar{x}$};
\end{scope}
		
\begin{scope}[shift={(5.5,0)}]
\draw[gray!50,dashed] (-0.99,-2.5) grid (5.5,2.5);
	
\foreach \x in {0,1,2,3,4}{
\draw (\x,0.05)--(\x,-0.05) node[anchor=north,xshift=-4pt] at (\x,0) {\small $\x$};}

\foreach \y in {-2,-1,1,2}
\draw (0.05,\y)--(-0.05,\y) node[left] {\small $\y$};
\draw[-{Stealth}] (-1,0)--(5.5,0) node[below right] {$x$};
\draw[-{Stealth}] (0,-2.5)--(0,2.5) node[above left] {$y$};
	
\draw[blue, very thick] (-1,0)--(2,0);
\node[blue, font=\small] at (-0.2,0.3) {$\text{\rm \textbf{A}}(F_2,\Omega_2,x_2^*)$};
\fill[blue](2,0) circle (1pt);

\clip (-1,-2.5) rectangle (5.5,2.5);
	
\draw[red, thick]
({3-7^(1/3)},1)--({3+7^(1/3)},1)--
plot[domain=1:-1,samples=200]
({3+(8-abs(\x)^3)^(1/3)},\x)--
({3-7^(1/3)},-1)--({3+7^(1/3)},-1)--
plot[domain=-1:1,samples=200]
({3-(8-abs(\x)^3)^(1/3)},\x)--cycle;
\node[red, font=\small] at (3,1.3) {$\text{\rm \textbf{E}}(F_2,\Omega_2,\bar x)$};

\fill(1,0) circle (2pt);
\node at (1.2,0.3) {$\bar{x}$};
\end{scope}
		
\end{tikzpicture}
\caption{Example~\ref{E7.2}\eqref{E7.2-2} with $p'=3$}
\label{fig8}
\end{figure}
\end{enumerate}
\end{example}	

\begin{example}
Let $X:=L^2[0,1]$ be the Hilbert space in Example~\ref{E4.9}.
Let $\mathbf{u},\mathbf{u}^*$, $x_1,x_2,\bx$ and $F_i,\Omega_i,\phi_i$ $(i=1,2)$ be given by \eqref{E4.9-1}--\eqref{E4.9-4}.
The objective function is given by \eqref{E7.2-3} for all $x\in X$.
By Example~\ref{E3.7}\eqref{E4.4.5}, $\bar x\in\text{\rm Sol}(P)$. 
The functions $\mathcal{T}^{F_1}_{\Omega_1}(x)$ and $\mathcal{T}^{F_2}_{\Omega_2}(x)$ are computed by \eqref{E4.13-1} and \eqref{E4.13-2} for all $x\in X$, respectively.
In particular,
$\mathcal{T}^{F_1}_{\Omega_1}(\bx)=\mathcal{T}^{F_2}_{\Omega_2}(\bx)=\frac{1}{2\sqrt 3}$.
Let $x^*_1:=\mathbf{u}$ and $x^*_2:=-\mathbf{u}$.
Then $x_i^*\in\partial\mathcal{T}^{F_i}_{\Omega_i}(\bx)$ $(i=1,2)$, and condition \eqref{T3.16-4} is satisfied.
The sets $\mathbf A(F_1,\Omega_1,x_1^*)$ and $\mathbf A(F_2,\Omega_2,x_2^*)$ are given by \eqref{E4.9-5}.
By \eqref{T6.1-1},
\begin{align*}
\text{\rm \textbf{E}}(F_i,\Omega_i,\bar x)
=\left\{x\in X\mid\mathcal{T}^{F_i}_{\Omega_i}(x)=\frac{1}{2\sqrt{3}}\right\}
\;\;(i=1,2).
\end{align*}	
Let $\lambda\in\R$.
Then
\begin{gather*}
\mathcal{T}^{F_1}_{\Omega_1}(\lambda\mathbf{u})=\left[\left|\lambda+\frac{1}{\sqrt 3}\right|-\frac{1}{2\sqrt 3}\right]^+,\;\;
\mathcal{T}^{F_2}_{\Omega_2}(\lambda\mathbf{u})=\left[\left|\lambda-\frac{1}{\sqrt 3}\right|-\frac{1}{2\sqrt 3}\right]^+,
\end{gather*}	
and consequently,
\begin{align*}
\mathcal{T}^{F_1}_{\Omega_1}(\lambda\mathbf{u})&=\frac{1}{2\sqrt{3}}\Leftrightarrow\left|\lambda+\frac{1}{\sqrt 3}\right|=\frac{1}{\sqrt{3}}
\Leftrightarrow\lambda\in\left\{-\frac{2}{\sqrt{3}},0\right\},\\
\mathcal{T}^{F_2}_{\Omega_2}(\lambda\mathbf{u})&=\frac{1}{2\sqrt{3}}\Leftrightarrow\left|\lambda-\frac{1}{\sqrt 3}\right|=\frac{1}{\sqrt{3}}
\Leftrightarrow\lambda\in\left\{\frac{2}{\sqrt{3}},0\right\}.
\end{align*}	
By \eqref{E4.9-5}, only $\lambda=0$ contributes to each intersection.
Thus,
\begin{align*}
\text{\rm \textbf{D}}(F_1,\Omega_1,x_1^*)
&=\text{\rm \textbf{A}}(F_1,\Omega_1,x_1^*)\cap \text{\rm\textbf{E}}(F_1,\Omega_1,\bar x)=\left\{\bx\right\},\\
\text{\rm \textbf{D}}(F_2,\Omega_2,x_2^*)&=\text{\rm \textbf{A}}(F_2,\Omega_2,x_2^*)\cap \text{\rm\textbf{E}}(F_2,\Omega_2,\bar x)
=\left\{\bx\right\}.
\end{align*}	
By \eqref{T3.16-5}, 
$\text{\rm Sol}(P)=\text{\rm \textbf{D}}(F_1,\Omega_1,x_1^*)\cap \text{\rm \textbf{D}}(F_2,\Omega_2,x_2^*)=\{\bx\}$.
\end{example}	

\section{Conclusions}\label{S8}

Dual necessary and sufficient optimality conditions are established for a general convex composition minimisation problem.
The entire solution set is characterised via a given solution and its dual data arising from the optimality conditions.
Sufficient conditions for the existence and boundedness of the solution set are discussed.

The results are applied to norm optimisation problems involving minimal time functions with closed convex target sets and bounded closed convex dynamics sets containing the origin in their interiors.
For the Fermat--Torricelli problem, the multipliers can all be chosen equal to one, and the solution set is represented as an intersection of sets on which fixed dual vectors remain subgradients.
Under additional projection assumptions, each set is decomposed into the sum of two sets described by target normal cones and dynamics Minkowski functionals.
For the Chebyshev centre problem, only indices with positive multipliers enter the solution set representation.
At every solution, the corresponding fixed dual vectors remain subgradients, and the associated minimal times attain the maximum.
For the $p$-Fermat--Torricelli problem, the dual data consist of
multipliers determined by the component minimal times and subgradients
of the component functions.
At every solution, each minimal time equals its value at the reference
solution, and the fixed dual vectors remain subgradients for components
with positive minimal times.

Possible directions for future research include extensions to directional and unbounded dynamics sets, stability of solution sets under perturbations of target and dynamics sets, computational methods based on the dual characterisations, and necessary optimality
conditions for nonconvex target and dynamics sets.

\section*{Acknowledgments}
The authors wish to thank Professor Alexander Kruger for  his comments  which helped us improve the manuscript.



\section*{Author contributions}
D.H.H. and N.D.C. contributed equally to this work.
Both authors read and approved the final manuscript.

\section*{Competing interests}
The authors declare that they have no competing interests.

\section*{Data availability}
Data sharing is not applicable to this article as no new data were created or analysed in this study.






\end{document}